\documentclass[final,leqno]{siamltex704}
\usepackage{hyperref}
\usepackage{amsmath}
\usepackage{graphicx}
\usepackage{array}
\usepackage{lineno}
\usepackage[notcite,notref]{showkeys}
\usepackage{tikz}
\usepackage{amsfonts,amssymb}
\usepackage{dsfont}
\usepackage{pifont}
\usepackage{subfigure}
\renewcommand{\ldots}{\dotsc}

\newtheorem{model-problem}{Problem}

\newcommand{\br}{\textbf{r}}

\newcommand{\bq}{\textbf{q}}

\def\T{{\mathcal T}}
\def\S{{\mathcal S}}

\def\E{{\mathcal E}}

\def\pT{{\partial T}}

\def\bn{{\bf n}}

\def\3bar{{|\hspace{-.02in}|\hspace{-.02in}|}}

\def\bbq{\begin{equation*}}
\def\eeq{\end{equation*}}
\def\br{\begin{eqnarray}}
\def\er{\end{eqnarray}}
\def\brr{\begin{eqnarray*}}
\def\err{\end{eqnarray*}}

\def\O{\Omega}
\def\OO{{\cal O}}

\def\E{{\mathcal E}}

\def\pa{\partial}

\def\bn{{\bf n}}

\def\3bar{{|\hspace{-.02in}|\hspace{-.02in}|}}

\newtheorem{WG}{WEAK GALERKIN ALGORITHM}
\newtheorem{SWG}{SIMPLIFIED WEAK GALERKIN ALGORITHM}
\begin{document}

\setlength{\parindent}{0.25in} \setlength{\parskip}{0.08in}

\title{Superconvergence of the Gradient Approximation for Weak Galerkin Finite Element Methods on Nonuniform Rectangular Partitions}

\author{
Dan Li\thanks{Department of Applied Mathematics, Northwestern Polytechnical University, Xi'an, Shannxi 710072, China. The research of Dan Li was supported in part by National Natural Science Foundation of China grant number 11471262.} \and
Chunmei Wang\thanks{Department of Mathematics, Texas State
University, San Marcos, TX 78666, USA. The research of Chunmei Wang was partially supported by National Science Foundation Awards DMS-1648171 and DMS-1749707.}
\and
Junping Wang\thanks{Division of Mathematical
Sciences, National Science Foundation, Alexandria, VA 22314
(jwang@nsf.gov). The research of Junping Wang was supported by the
NSF IR/D program, while working at National Science Foundation.
However, any opinion, finding, and conclusions or recommendations
expressed in this material are those of the author and do not
necessarily reflect the views of the National Science Foundation.}}

\maketitle
%\linenumbers

\begin{abstract}
This article presents a superconvergence for the gradient approximation of the second order elliptic equation discretized by the weak Galerkin finite element methods on nonuniform rectangular partitions. The result shows a convergence of ${\cal O}(h^r)$, $1.5\le r \le 2$, for the numerical gradient obtained from the lowest order weak Galerkin element consisting of piecewise linear and constant functions. For this numerical scheme, the optimal order of error estimate is ${\cal O}(h)$ for the gradient approximation. The superconvergence reveals a superior performance of the weak Galerkin finite element methods. Some computational results are included to numerically validate the superconvergence theory.
\end{abstract}



\section{Introduction}

Superconvergence is a phenomenon in numerical methods in which approximate solutions converge to the exact solution of the problem at rates higher than the optimal order as measured globally against polynomial interpolations. Superconvergence often occurs at
particular locations of low dimension such as points or lines/curves for two dimensional problems. One of the main tasks in the study of superconvergence is to identify the area (i.e., discrete set of points or lines/curves) where the numerical solutions have superior performance. In scientific computing, superconvergence has been used to yield new approximations with improved and/or prescribed accuracies through postprocessing techniques involving relatively small amount of computation. Superconvergence has also been employed to provide guiding principles or posteriori error estimators for adaptive grid refinement strategies \cite{sup_CMAME1992,sup_SJSC2014,supposter_PIJNM11992,supposter_PIJNM21992,STOKES_NA2017}. Superconvergence has played a significant role in high performance computing ever since its first discovery in the seventies of the last century \cite{dd1973, bs1977}, and the area remains to be a very active branch of numerical partial differential equations.

In this paper, we are concerned with new developments of superconvergence for weak Galerkin finite element approximations of boundary value problems (BVP). For simplicity, we consider the second order elliptic equation that seeks an unknown function $u\in H^1(\Omega)$ satisfying
\begin{eqnarray}%\label{model}
-\nabla \cdot (a\nabla u) &=& f, \quad \mbox{in}~~ \O, \label{a1}\\
u &=& g,\quad \mbox{on}~~ \pa\O , \label{aa1}
\end{eqnarray}
where $\Omega$ is an open bounded domain in $\mathbb{R}^2$ with Lipschitz continuous boundary $\partial \Omega$; $f=f(x,y)\in H^{-1}(\Omega)$ and $g=g(x,y)\in H^{\frac12}(\partial\Omega)$ are given functions defined on $\Omega$ and its boundary $\partial\Omega$, respectively. We assume that the diffusive coefficient tensor $a=\{a_{ij}\}_{2\times 2}$ is symmetric, uniformly bounded and positive definite in $\O$. Here and in what follows of this paper, we adopt the usual notation of Sobolev spaces used in \cite{ciarlet-fem,gr}.

Numerical solutions to \eqref{a1}-\eqref{aa1} can be obtained by using a variety of computational methods, including the finite difference, the finite volume, the collocation, and the finite element methods. For the superconvergence study, we shall focus on the weak Galerkin finite element method (WG-FEM) developed in \cite{ellip_JCAM2013, ellip_MC2014, mwy-pkpk-1, mwy} (see also the references cited therein). WG-FEM is a relatively new numerical method for partial differential equations. The method is technically a generalization of the classical Galerkin finite element method \cite{ciarlet-fem,gr} through a relaxed assumption on the smoothness of the approximating functions. WG-FEM has three fundamental ingredients in its formulation: (1) {\em conventional weak form} - it is based on the usual variational formulation for the PDE problem; (2) {\em weak derivatives} - it makes use of weak (often discontinuous) finite element functions for which generalized weak partial derivatives are introduced by mimicking the definition of distributions; and (3) {\em weak continuity} - parameter-independent stabilizers are employed to ensure weak continuity of the numerical solution. WG-FEM has advantages over the classical Galerkin finite element method in several aspects. First, the approximating functions are flexible and easy to represent as WG is based on piecewise polynomials with great flexibility in the continuity requirement. Next, the finite element partitions are allowed to contain polygons or polyhedra of arbitrary shape in WG-FEM \cite{ellip_MC2014} so that the method is robust with respect to domain partitioning. Furthermore, it is known that WG schemes are absolutely stable, and the corresponding solutions generally preserve the physical quantities inherited by the modeling equations at discrete levels. WG-FEM has been developed for many PDEs including the linear elasticity equation \cite{sup_EP2016}, the Stokes equation \cite{STOKES_NA2017}, Maxwell's equation \cite{MAXWELL_JSC2015,CWang-Maxwell}, the elliptic interface problem \cite{EIP_NA2013}, the Brinkman equation \cite{WG_BRINK2014}, the Helmholtz equation \cite{WG_HEL2014}, the Sobolev equation \cite{PDWG_SE2017}, and the wave equation \cite{parabolicWG_JSC2017} etc. The latest development of WG-FEM is the primal-dual weak Galerkin finite element method for second order elliptic equations in non-divergence form \cite{PDWG_MC2017} and the Fokker-Planck equation \cite{fp}.

Some superconvergence has been observed for weak Galerkin finite element approximations in published and unpublished numerical experiments. In \cite{ellip_NA2013,parabolicWG_JSC2017}, the authors have noted some superconvergence for the gradient approximation in their numerical experiments for the elliptic and hyperbolic equations (e.g., $\|\nabla_d e_h\|$ in Table 4.11 \cite{ellip_NA2013}). In \cite{PDWG_MC2017}, a superconvergence of order ${\cal O}(h^4)$ was observed for the approximation of the primal variable when piecewise quadratic functions are employed in the numerical scheme. To the author's knowledge, no mathematical theory has been derived for the superconvergent results observed numerically in \cite{ellip_NA2013,parabolicWG_JSC2017,PDWG_MC2017}. The goal of this paper is to establish a mathematical theory for some of these numerical observations.

A vast amount of literature now exists on superconvergence for various numerical methods. For the classical Galerkin finite element method, it is well known that superconvergence often occurs when the governing equations have smooth solutions and are approximated by finite element schemes on partitions with special properties such as uniformity, local point symmetry, local translation invariance, and orthogonality (e.g., rectangular partitions). Many results on superconvergence have been derived in the last four decades in the finite element context, for example classical finite element method \cite{sup_SIAM1991, sup_HUNAN1995,sup_GFEM2006,sup_LSS2000,CWang-SuperC}, discontinuous Galerkin method \cite{sup_SJNA2013, sup_SJNAp1997}, hybridizable discontinuous Galerkin method \cite{LUHDG_SIAM2009,SHDG_SIAM2016}, smoothed finite element method \cite{sup_CRC2016}. For weak Galerkin finite element method, Harris \cite{parabolic_AMC2014} derived a superconvergence for elliptic equations by using the $L^{2}$-projection technique introduced in \cite{sup_LSS2000}. It was shown in \cite{parabolic_AMC2014} that the  projected numerical solution of the lowest order is convergent to the exact solution at the rate of $\OO(h^{1.5})$ or better in the usual $H^1$-norm. In \cite{RuishuWang}, a post-processing technique using the polynomial preserving recovery (PPR) was introduced for the WG
approximation arising from schemes with bi-polynomials and over-penalized stabilization terms on uniform rectangular partitions.

The main contribution of this paper is the establishment of an $\OO(h^r)$, $1.5\le r\le 2$, error estimate for the gradient approximation of the model problem \eqref{a1}-\eqref{aa1} when discretized by the lowest order WG-FEM on nonuniform rectangular partitions. The lowest order WG element consists of piecewise linear functions on each element plus piecewise constant functions on each element boundary. The discrete weak gradient is computed as a piecewise constant vector-valued function. For the lowest order WG-FEM under consideration, the optimal order of convergence for the gradient approximation is known to be $\OO(h)$, so that the convergence of order $\OO(h^r)$ reveals a super performance of the corresponding WG-FEM.

The rest of the paper is organized as follows. In Section \ref{Section:WeakGradient}, we briefly review the definition and the computation for weak gradients. In Section \ref{Section:WG}, we present a detailed description of the WG-FEM. Section \ref{Section:SWG} is devoted to the derivation of a simplified formulation for the WG-FEM.
In Section \ref{Section:EE}, we derive error equations for the simplied WG-FEM. In Section \ref{Section:superC}, we carry out a superconvergence analysis in great details. Finally in Section \ref{Section:NE}, we report some numerical results to verify the superconvergence theory.

\section{Weak Gradient and Discrete Weak Gradient}\label{Section:WeakGradient}

This section aims to review preliminaries for the weak Galerkin
finite element method; namely, the discrete weak gradient operator
introduced in \cite{ellip_JCAM2013}.

Throughout the paper, we use the standard notations for Sobolev spaces and norms \cite{ciarlet-fem,gr}. For any open set $D\subset\mathbb{R}^{2}$, $\|\cdot\|_{s,D}$ and $(\cdot,\cdot)_{s,D}$ denote the norm and inner-product in the Sobolev space $H^s(D)$ consisting of square integrable partial derivatives up to order $s$. When $s=0$ and $D=\Omega$, we shall drop the subscripts in the norm and inner-product notation.

Let $T$ be any polygonal domain with boundary $\partial T$. By a weak
function on $T$ we mean $v=\{v_0,v_b\}$ where $v_0\in L^2(T)$ and $v_b\in L^2(\partial T)$. The first component $v_0$
represents the value of $v$ in the interior of $T$, and the second component $v_b$ is the value of $v$ on $\partial T$. We emphasize that
$v_b$ may not be related to the trace of $v_0$ on
$\partial T$, should a trace be well-defined, though $v_b= v_0|_{\partial T}$ is a viable choice in the algorithm design.

Denote by $W(T)$ the space of all weak functions on $T$:
\begin{equation*}
W(T)=\{v=\{v_0,v_b\}: v_0\in L^2(T), v_b\in L^2(\partial T)\}.
\end{equation*}
The weak gradient operator is denoted by $\nabla_w$ from $W(T)$ to the dual space of $[H^1(T)]^2$ whose action on each weak function $v\in W(T)$ is given by
\begin{equation}\label{2.3-2}
\langle \nabla_{w} v,\boldsymbol{\psi}\rangle_T=-(v_0, \nabla \cdot \boldsymbol{ \psi})_T+
\langle v_b ,\boldsymbol{ \psi}\cdot  \textbf{n}\rangle_{\partial T},\qquad \forall \boldsymbol{\psi} \in [H^1(T)]^2.
\end{equation}
Here the left-hand side of \eqref{2.3-2} denotes the action of $\nabla_{w} v$ on $\boldsymbol{\psi} \in [H^1(T)]^2$, and $\bn$ is the unit outward normal vector to $\partial T$.
%\end{definition}

For the sake of computation, the weak gradient operator $\nabla_w$
must be discretized in one way or another. In this paper, we use polynomials to approximate the weak gradient. More precisely, for
any given non-negative integer $r\geq 0$, let $P_r(T)$ be the set
of polynomials on $T$ with total degree $r$ or less, which is used to approximate the weak gradient.

\begin{definition}
The discrete weak gradient operator, denoted by $\nabla_{w,r,T}$, is defined as a linear operator so that for any $v\in W(T)$, the action $\nabla_{w,r,T} v$ is the unique vector-valued polynomial in $[P_r(T)]^2$ satisfying
\begin{equation}\label{2.4-2}
(\nabla_{w,r,T}v,\boldsymbol{ \psi})_T=-(v_0, \nabla \cdot
\boldsymbol{ \psi})_T+ \langle v_b ,\boldsymbol{ \psi}\cdot
\textbf{n}\rangle_{\partial T},\quad \forall \boldsymbol{ \psi}\in
[P_r(T)]^2.
\end{equation}
\end{definition}

\section{Algorithm of Weak Galerkin}\label{Section:WG}

Let ${\cal T}_h$ be a polygonal partition of the domain $\Omega$ that is shape regular as defined in \cite{ellip_MC2014}. Denote by
$\E_h$ the set of all edges in ${\cal T}_h$, and
$\E_h^0=\E_h\setminus\partial\Omega$ the set of all interior
edges. Let $h_T$ be the diameter of $T\in {\cal
T}_h$ and $h=\max_{T\in {\cal T}_h}h_T$ the mesh size of the
partition ${\cal T}_h$.

Let $k\ge 1$ be a given positive integer. On each element $T\in
{\cal T}_h$, we introduce a local weak finite element space $V(T, k)$ as follows
$$
V(T,k)=\{v =\{v_0, v_b\}: \ v_0\in P_{k}(T), v_b\in P_{k-1}(e),\
e\subset\pT\}.
$$
By patching the local elements $V(T,k)$ through a common value $v_b$
on the interior edges $\E_h^0$, we have a global weak
finite element space
$$
V_h=\{v =\{v_0,v_b\}: \ v|_T\in V(T,k),\ \mbox{$v_b$ is single valued
on $\E_h$}\}.
$$
Denote by $V_h^0$ the subspace of $V_h$ consisting of the finite element functions with vanishing boundary value; i.e.,
$$
V_h^0=\{v=\{v_0,v_b\} \in V_h: \ v_b|_{\partial \Omega}=0\}.
$$

The discrete weak gradient $\nabla_{w,k-1}v$ for $v \in V_h$ is computed by using
vector-valued polynomials of degree $k-1$ on each element $T\in
T_h$; namely,
\begin{eqnarray*}
(\nabla_{w,k-1}v)|_T &=& \nabla_{w,k-1,T}(v|_T), \qquad v\in V_h.
\end{eqnarray*}
For simplicity, we shall drop the subscript $k-1$ from
the discrete weak gradient operator notation $\nabla_{w,k-1}$, and use
$\nabla_d$ to denote $\nabla_{w,k-1}$; i.e.,
$$
\nabla_d v :=\nabla_{w,k-1}v,\qquad v\in V_h.
$$

Next, we introduce two bilinear forms in $V_h\times V_h$
\begin{eqnarray*}
(a\nabla_d w, \nabla_d v)_h &=& \sum_{T\in{\cal T}_h} (a\nabla_d w,
\nabla_d v)_T,\\
s(w,v)&=&\rho h^{-1}\sum_{T\in {\cal T}_h} \langle w_b-Q_b w_0, v_b-Q_b
v_0\rangle_{\partial T},
%s(w,v)&=&\rho h^{-1}\sum_{T\in {\cal T}_h} \langle Q_b w_0-w_b, Q_b
%v_0-v_b\rangle_{\partial T},
\end{eqnarray*}
where $\rho > 0$ is any parameter, $Q_b$ is the usual $L^2$
projection operator from $L^{2}(\pT)$ to $P_{k-1}(\pT)$.

%\begin{algorithm}
\begin{WG}
A numerical approximation for \eqref{a1}-\eqref{aa1} can be obtained by seeking $u_{h}=\{u_{0},u_{b}\}\in V_{h}$, such that $u_b=\widetilde{Q}_b g$ on $\partial \Omega$ and satisfying
\begin{equation}\label{WG-scheme}
(a\nabla_d u_{h}, \nabla_d v)_h+s(u_{h},v)=(f,v_0),\qquad \forall v\in V_h^0,
\end{equation}
where $\widetilde{Q}_b g$ is a suitably chosen projection of the Dirichlet boundary data by using polynomials of degree $k-1$.
%\end{algorithm}
\end{WG}

The approximate boundary date $\widetilde{Q}_b g$ may take the following form:
\begin{equation}\label{boundary-approximation}
\widetilde{Q}_b g:=Q_b g + \varepsilon_b,
\end{equation}
where $\varepsilon_b$ is viewed as a small perturbation of the $L^2$ projection $Q_b g$. A typical example of the perturbation term is given by $\varepsilon_b=0$ so that $\widetilde{Q}_b g=Q_b g$. But it will be seen later in Section \ref{Section:superC} that a non-zero perturbation is necessary for a superconvergence of the weak gradient.

Note that the system \eqref{WG-scheme} is symmetric and positive definite for any parameter value $\rho>0$; i.e., the system \eqref{WG-scheme} is solvable.

\section{Simplified WG Formulation}\label{Section:SWG}

For superconvergence, we shall study the lowest order
WG finite element; i.e., $k=1$ in the numerical scheme
(\ref{WG-scheme}). Thus, the finite element
approximation $u_h$ is a piecewise linear function in the interior
and piecewise constant on the element boundary. The discrete weak gradient $\nabla_d u_h$ is a vector-valued polynomial in $[P_{0}(T)]^{2}$.

Note that any $v=\{v_0, v_b\}\in V_h$ can be decomposed as follows
$$
\{v_0, v_b\}=\{v_0, 0\}+\{0, v_b\},
$$
which, for simplicity of notation, shall be denoted as $v=v_{0}+v_{b}$.
Denote by $V_b=\{ v_{b}=\{0, v_b\}\in V_h\}$ the boundary space and
$V_0=\{ v_{0}=\{v_0, 0\}\in V_h\}$ the interior space, respectively.
Using the definition of the discrete weak gradient, it is not hard to see from (\ref{2.4-2}) that $\nabla_d v_{0}=0$ for any $v_0\in V_0$. It follows that
the weak Galerkin algorithm (\ref{WG-scheme}) can be reformulated as
follows: Find $u_{h}\in V_h$ such that $u_b=\widetilde{Q}_b g$ on $\partial \Omega$ and satisfying
\begin{equation}\label{al2}
(a\nabla_d u_b, \nabla_d v_b)_h+s(u_{h},v)=(f,v_0), \qquad \forall  v\in
V_h^0.
\end{equation}

Next, we introduce an extension operator $\S$ that maps
$v_b\in P_0(\pT)$ to a linear function on $T$ such that
\begin{equation}\label{EQ:extension-S:new}
\langle {\S}(v_b), Q_b \phi \rangle_{\partial T} =\langle v_b,
\phi \rangle_{\partial T}, \qquad \forall \phi\in P_1(T).
\end{equation}
Thus, on each element $T$ we have
\begin{eqnarray*}
\langle u_b- Q_b u_0, v_b - Q_b {\S}(v_b)\rangle_{\partial T} & = &
\langle u_b, v_b - Q_b {\S}(v_b)\rangle_{\partial T} \\
&=& \langle u_b- Q_b {\S}(u_b), v_b - Q_b {\S}(v_b)\rangle_{\partial T}.
\end{eqnarray*}
From the above identity, for any $v=\{\S(v_b), v_b\}\in V_h^{0}$ we have
\begin{equation}
\begin{split}
s(u_h,v) & =\rho h^{-1} \sum_{T\in\T_h} \langle u_b-Q_b u_0, v_b-Q_b {\S}(v_b)\rangle_{\partial T}\\
& = \rho h^{-1}\sum_{T\in\T_h} \langle u_b- Q_b {\S}(u_b), v_b-Q_b {\S}(v_b)\rangle_{\partial T}.
\end{split}
\end{equation}
By letting $v=\{{\S}(v_b), v_b\}\in V_h^0$ in (\ref{al2}) we arrive at the following {\em simplified weak Galerkin} method.

\begin{SWG}
Find $u_b\in V_b^g$ such that
\begin{equation}\label{al3}
%\begin{split}
 \sum_{T\in {\cal T}_h}(a\nabla_d u_b, \nabla_d v_b)_T+ \rho h^{-1}\sum_{T\in {\cal T}_h}
\langle u_b-Q_b{\S}(u_b), v_b-Q_b{\S}(v_b) \rangle_{\partial T}
= (f,  {\S}(v_b))
%\end{split}
\end{equation}
for all $v_b\in V_b^0$, where $V_b^0=\{v_b\in V_b: v_b|_{\partial \Omega}=0\}$ and $V_b^g=\{v_b\in V_b: v_b|_{\partial \Omega}=\tilde Q_b g\}$.
\end{SWG}

For simplicity of analysis, we assume that the coefficient tensor $a$ in \eqref{a1} is a piecewise constant matrix with respect to the finite element partition ${\cal T}_h$. The result can be extended to variable coefficient tensors without any difficulty, provided that the tensor $a$ is piecewise smooth. For simplicity of notation, we introduce a flux variable $\textbf{q}=a\nabla u$.

\section{Error Equations}\label{Section:EE}

The goal of this section is to derive an error equation for the
simplified weak Galerkin scheme (\ref{al3}). To this end, let $\mathbb{Q}_{h}$ be the standard $L^2$ projection operator onto the local discrete gradient space $[P_{0}(T)]^{2}$. On each element $T\in\T_h$, the following commutative property holds true  \cite{ellip_JCAM2013}:
\begin{equation}\label{EQ:q0}
\nabla_d Q_b w = \mathbb{Q}_{h} \nabla w,\qquad w\in H^1(T).
\end{equation}
Denote by $e_b=Q_bu-u_b$ the error function between the WG solution and the
$L^2$ projection of the exact solution of the model problem
\eqref{a1}-\eqref{aa1}.

\begin{lemma}
 The error function $e_b=Q_bu-u_b$ satisfies the following equation
\begin{equation}\label{error}
\sum_{T\in {\cal T}_h} (a\nabla_d e_b, \nabla_d v_b)_T
+\rho h^{-1}\sum_{T\in {\cal T}_h}\langle e_b-Q_b{\S}(e_b), v_b-Q_b{\S}(v_b)\rangle_{\partial T}\\
=  \zeta_u (v_b),
\end{equation}
for all $v_{b}\in V_b^0$, where
\begin{equation}\label{EQ:functional}
\begin{split}
\zeta_u(v_b) = & \sum_{T\in {\cal T}_h} \langle {\color{black}{(\textbf{q}
 -\mathbb{Q}_h \textbf{q}) \cdot \bn}}, {\S}(v_b)-v_b\rangle_{\partial T} \\
& \ + \rho h^{-1}{\color{black}{\sum_{T\in {\cal T}_h}}}\langle {\color{black}{Q_b{\S}(Q_bu)-Q_bu}}, {\color{black}{Q_b{\S}(v_b)-v_b}} \rangle_{\partial T}
\end{split}
\end{equation}
is a linear functional on $V_b$.
\end{lemma}

\begin{proof}
Using (\ref{EQ:q0}), (\ref{2.4-2}) with $\boldsymbol\psi={\mathbb{Q}_h \textbf{q}}$, and
the usual integration by parts, we obtain
\begin{equation*}
\begin{split}
 \sum_{T\in {\cal T}_h} (a\nabla_d Q_b u, \nabla_d v_b)_T
 =& \sum_{T\in {\cal T}_h} (a \mathbb{Q}_h \nabla u, \nabla_d v_b)_T \\
 =& \sum_{T\in {\cal T}_h} ({\mathbb{Q}_h \textbf{q}}, \nabla_d v_b)_T \\
=& \sum_{T\in {\cal T}_h} ({\mathbb{Q}_h \textbf{q}}, \nabla {\S}(v_b))_T+{\sum_{T\in {\cal T}_h}}\langle {\mathbb{Q}_h \textbf{q}}\cdot \bn, v_b-{\S}(v_b)\rangle_{\partial T}\\
= &\sum_{T\in {\cal T}_h} ({\color{black}{\textbf{q}}}, \nabla {\S}(v_b))_T+{\color{black}{\sum_{T\in {\cal T}_h}}}\langle{\color{black}{\mathbb{Q}_h \textbf{q}}}\cdot \bn, v_b-{\S}(v_b)\rangle_{\partial T}\\
= &\sum_{T\in {\cal T}_h} (-{\color{black}{\nabla\cdot  \textbf{q}}},  {\S}(v_b))_T+{\color{black}{\sum_{T\in {\cal T}_h}}}\langle  {\color{black}{\textbf{q}}}\cdot \bn, {\S}(v_b)\rangle_{\partial T}\\&+{\color{black}{\sum_{T\in {\cal T}_h}}}\langle{\color{black}{\mathbb{Q}_h \textbf{q}}}\cdot \bn, v_b-{\S}(v_b)\rangle_{\partial T}\\
= &(f,  {\S}(v_b)) + \sum_{T\in {\cal T}_h} \langle {\color{black}{ \textbf{q}}}\cdot \bn, {\S}(v_b)-v_b\rangle_{\partial T}+{\color{black}{\sum_{T\in {\cal T}_h}}}\langle {\color{black}{\mathbb{Q}_h \textbf{q}}}\cdot \bn, v_b-{\S}(v_b)\rangle_{\partial T}\\
= &(f,  {\S}(v_b)) + \sum_{T\in {\cal T}_h} \langle  {\color{black}{(\textbf{q}-\mathbb{Q}_h \textbf{q})}}  \cdot \bn), {\S}(v_b)-v_b\rangle_{\partial T},
\end{split}
\end{equation*}
where we have also used ${\color{black}{-\nabla \cdot \textbf{q}=f}}$ and $\sum_{T\in {\cal T}_h}
\langle {\color{black}{\textbf{q}}} \cdot \bn, v_b\rangle_{\partial T}=0$ as is single-valued on $\E_h^0$ and $v_b=0$ on
$\partial \Omega$. Thus, from the simplified WG algorithm (\ref{al3}) we obtain
\begin{equation*}
\begin{split}
 &\sum_{T\in {\cal T}_h} ({\color{black}{a}}\nabla_d (Q_b u-u_b), \nabla_d v_b)_T
 +\rho h^{-1} {\color{black}{\sum_{T\in {\cal T}_h}}}\langle e_b-Q_b{\S}(e_b), v_b-Q_b{\S}(v_b)\rangle_{\partial T}\\
=&\ (f,  {\S}(v_b))
+ \sum_{T\in {\cal T}_h} \langle  {\color{black}{(\textbf{q}-\mathbb{Q}_h \textbf{q})}} \cdot \bn, {\S}(v_b)-v_b\rangle_{\partial T}\\
 & -(f,  {\S}(v_b)) +{\color{black}{\rho h^{-1}}}\sum_{T\in {\cal T}_h} \langle Q_bu-Q_b{\S}(Q_bu), v_b-Q_b{\S}(v_b) \rangle_{\partial T}\\
=&  \sum_{T\in {\cal T}_h} \langle  {\color{black}{(\textbf{q}-\mathbb{Q}_h \textbf{q})}} \cdot \bn, {\S}(v_b)-v_b\rangle_{\partial T} + \rho h^{-1}
{\color{black}{\sum_{T\in {\cal T}_h}}}\langle Q_bu-Q_b{\S}(Q_bu), v_b-Q_b{\S}(v_b)
\rangle_{\partial T}.
 \end{split}
\end{equation*}
This completes the proof of the lemma.
\end{proof}

\section{Superconvergence on Rectangular Elements}\label{Section:superC}

Consider the model problem (\ref{a1})-(\ref{aa1}) defined on rectangular domains. For simplicity, let the domain be given by $\Omega = (0,1)^2$ which is partitioned into rectangular elements as the Cartesian product of two partitions $\Delta_x$ and $\Delta_y$ for the
unit interval $I=(0,1)$:
\begin{eqnarray*}
\Delta_x:\ && 0=x_0< x_1<x_2<\ldots<x_{n-1}<x_n=1,\\
\Delta_y:\ && 0=y_0< y_1<y_2<\ldots<y_{m-1}<y_m=1.
\end{eqnarray*}
The solution of this model problem can be approximated by using the simplified weak Galerkin finite element scheme \eqref{al3}. The goal of this section is to study the accuracy or superconvergence of the numerical solutions when the lowest order (i.e., $k=1$) of element is employed.

\subsection{Some technical results}
Each rectangular element $T\in\T_h$ can be represented as
$T=[x_{i-1}, x_{i}]\times [y_{j-1}, y_{j}]$ for $i\in \{1, 2,
\cdots, n\}$ and $j\in \{1, 2, \cdots, m\}$. Figure
\ref{rectangular-element} depicts a typical rectangular element
under consideration. For convenience, the length of edge $e_s$ is
denoted as $|e_s|$ so that $|e_1|=|e_2|$ and $|e_3|=|e_4|$, and the
mid-point of edge $e_s$ is denoted as $M_s=(x_s^*, y_s^*),
s=1,\ldots, 4$.

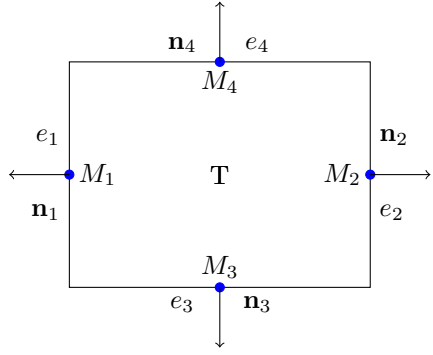
\begin{figure}[h]
\begin{center}
\begin{tikzpicture}
\coordinate (A1) at (-2,0); %\draw node[left] at (A1) {$A_1$};
\coordinate (A2) at (2, 0); %\draw node[right] at (A2) {$A_2$};
\coordinate (A3) at (2, 3); %\draw node[right] at (A3) {$A_3$};
\coordinate (A4) at (-2,3); %\draw node[left] at (A4) {$A_4$};
\coordinate (M1) at (-2,1.5);
\coordinate (M2) at (2,1.5);
\coordinate (M3) at (0,0);
\coordinate (M4) at (0,3);
\coordinate (M1end) at (-2.8,1.5);
\coordinate (M2end) at (2.8,1.5);
\coordinate (M3end) at (0.0,-0.8);
\coordinate (M4end) at (0,3.8);
\draw node[right] at (M1) {$M_1$};
\draw node[left] at (M2) {$M_2$};
\draw node[above] at (M3) {$M_3$};
\draw node[below] at (M4) {$M_4$};
\coordinate (center) at (0,1.5);
\coordinate (e1) at (-2,2.0);
\coordinate (e2) at (2,1.0);
\coordinate (e3) at (-0.5,0);
\coordinate (e4) at (0.5,3);
\coordinate (ne1) at (-2,1.0);
\coordinate (ne2) at (2,2.0);
\coordinate (ne3) at (0.5,0);
\coordinate (ne4) at (-0.5,3);
\draw node[left] at (e1) {$e_1$};
\draw node[right] at (e2) {$e_2$};
\draw node[below] at (e3) {$e_3$};
\draw node[above] at (e4) {$e_4$};
\draw node[left] at (ne1) {$\bn_1$};
\draw node[right] at (ne2) {$\bn_2$};
\draw node[below] at (ne3) {$\bn_3$};
\draw node[above] at (ne4) {$\bn_4$};
\draw node at (center) {T}; \draw node
at (center) {T}; \draw (A1)--(A2)--(A3)--(A4)--cycle;
\filldraw[blue] (M1) circle(0.06);
\filldraw[blue] (M2) circle(0.06);
\filldraw[blue] (M3) circle(0.06);
\filldraw[blue] (M4) circle(0.06);
\draw[->] (M1)--(M1end);
\draw[->] (M2)--(M2end);
\draw[->] (M3)--(M3end);
\draw[->] (M4)--(M4end);
%\draw (A2)--(A4);
%\draw[dashed] (A1)--(A3);
%\draw[->] (n3)--(n3end) node[above]{$\mathbf{n}(3)$};
%\draw[->] (n4)--(n4end) node[above]{$\mathbf{n}(1)$};
%   \filldraw[black] (A1) circle(0.1);
%    \filldraw[black] (A2) circle(0.1);
%    \filldraw[black] (A3) circle(0.1);
%    \filldraw[black] (A4) circle(0.1);
\end{tikzpicture}
\caption{Depiction of a rectangular element
$T\in\T_h$}\label{rectangular-element}
\end{center}
\end{figure}

On the element $T$, denote by $v_{bs}$ the value of $v_b$ on the edge $e_s, \ s=1,\ldots, 4$. From (\ref{2.4-2}) for the weak gradient, we have
$$
(\nabla_d v_b, \boldsymbol \psi)_T = \langle v_b, \boldsymbol\psi\cdot\bn\rangle_\pT,\qquad \forall \boldsymbol\psi \in [P_0(T)]^2,
$$
which leads to the following formulation
\begin{equation}\label{EQ:weak-gradient}
\nabla_d v_b = \left(\frac{v_{b2}-v_{b1}}{|e_3|},\frac{v_{b4}-v_{b3}}{|e_1|}\right)^\prime.
\end{equation}

Recall that the extension function ${\S}(v_b)\in P_1(T)$ is defined by the equation (\ref{EQ:extension-S:new}) with
linear test function $\phi$ so that $Q_b\phi = \phi(M_s)$ on each edge $e_s$. Thus, the equation (\ref{EQ:extension-S:new}) can be rewritten as
\begin{equation}\label{svb:001}
\sum_{s=1}^4 |e_s| {\S}(v_b) (M_s) \phi(M_s) =\sum_{s=1}^4 |e_s|
v_{bs} \phi(M_s),\qquad \forall \phi\in P_1(T).
\end{equation}

\begin{lemma}\label{midvalue}
Let $v_b$ be given on the element $T=[x_{i-1}, x_i]\times [y_{j-1},
y_j]$, and ${\S}(v_b)\in P_1(T)$ be the extension function of $v_b$ in
$T$ defined by (\ref{svb:001}). Then, the following {\color{black}{results hold}}
true
\begin{equation}\label{EQ:S-Property:002}
\begin{split}
({\S}(v_b)-v_b)(M_1) & =({\S}(v_b)-v_b)(M_2)\\
& = \frac{|e_3|}{2(|e_1|+|e_3|)}(v_{b3}+v_{b4}-v_{b1}-v_{b2}),\\
\end{split}
\end{equation}
\begin{equation}\label{EQ:S-Property:008}
\begin{split}
({\S}(v_b)-v_b)(M_3) & =({\S}(v_b)-v_b)(M_4)\\
& =- \frac{|e_1|}{2(|e_1|+|e_3|)}(v_{b3}+v_{b4}-v_{b1}-v_{b2}).
\end{split}
\end{equation}
Hence,
\begin{equation}\label{EQ:S-Property:021}
|e_1|({\S}(v_b)-v_b)(M_1) = - |e_3| ({\S}(v_b)-v_b)(M_3).
\end{equation}
\end{lemma}

\begin{proof} Let $(x_c, y_c)$ be the center of the element $T$, and
assume
$$
{\S}(v_b)=c_1+c_2(x-x_c)+c_3(y-y_c).
$$
As ${\S}(v_b)$ satisfies (\ref{svb:001}), we may choose $\phi=1$
in (\ref{svb:001}) to obtain
\begin{equation*}
\begin{split}
\sum_{s=1}^4 |e_s| \Big( c_1+c_2(x_s^*-x_c)+c_3(y_s^*-y_c)
\Big)=\sum_{s=1}^4 |e_s|v_{bs},
\end{split}
\end{equation*}
which leads to
$$
c_1=\frac{|e_1|(v_{b1}+v_{b2})+|e_3|(v_{b3}+v_{b4})}{2|e_1|+2|e_3|}.
$$
Next, by letting $\phi=x-x_c$ in (\ref{svb:001}) we obtain
\begin{equation*}
\begin{split}
\sum_{s=1}^4 |e_s|  \Big( c_1+c_2(x_s^*-x_c)+c_3(y_s^*-y_c) \Big)
(x_s^*-x_c)=\sum_{s=1}^4 |e_s| v_{bs} (x_s^*-x_c),
\end{split}
\end{equation*}
which gives rise to
$$
c_2=\frac{v_{b2}-v_{b1}}{|e_3|}.
$$
Analogously, by letting $\phi=y-y_c$ in (\ref{svb:001}) we arrive at
$$
c_3=\frac{v_{b4}-v_{b3}}{|e_1|}.
$$
It follows that
$$
{\S}(v_b)=\frac{|e_1|(v_{b1}+v_{b2})+|e_3|(v_{b3}+v_{b4})}{2(|e_1|+|e_3|)}+\frac{v_{b2}-v_{b1}}{|e_3|}(x-x_c)
+\frac{v_{b4}-v_{b3}}{|e_1|}(y-y_c).
$$

Next, we will compute the value of ${\S}(v_b)-v_b$ at the
midpoint of each edge. At the midpoint $M_1=(x_1^*, y_1^*)$ of the edge $e_1$, we have
\begin{equation*}
\begin{split}
({\S}(v_b)-v_b)|_{M_1}=&\frac{|e_1|(v_{b1}+v_{b2})+|e_3|(v_{b3}+v_{b4})}{2(|e_1|+|e_3|)}- \frac{|e_3|}{2}
\frac{v_{b2}-v_{b1}}{|e_3|}- v_{b1} \\
=&   \frac{|e_3|}{2(|e_1|+|e_3|)}(v_{b3}+v_{b4}-v_{b1}-v_{b2}).\\
\end{split}
\end{equation*}
At the midpoint $M_2$ of the edge $e_2$, we have
\begin{equation*}
\begin{split}
({\S}(v_b)-v_b)|_{M_2}=&\frac{|e_1|(v_{b1}+v_{b2})+|e_3|(v_{b3}+v_{b4})}{2(|e_1|+|e_3|)}
+\frac{|e_3|}{2} \frac{v_{b2}-v_{b1}}{|e_3|}- v_{b2} \\
=& \frac{|e_3|}{2(|e_1|+|e_3|)}(v_{b3}+v_{b4}-v_{b1}-v_{b2}).\\
\end{split}
\end{equation*}
Similarly, at the midpoint of the edges $e_3$ and $e_4$, we have

\begin{equation*}
\begin{split}
({\S}(v_b)-v_b)|_{M_3}
=&-\frac{|e_1|}{2(|e_1|+|e_3|)}(v_{b3}+v_{b4}-v_{b1}-v_{b2}).\\
\end{split}
\end{equation*}

$$
({\S}(v_b)-v_b)|_{M_4}=
-\frac{|e_1|}{2(|e_1|+|e_3|)}(v_{b,3}+v_{b,4}-v_{b,1}-v_{b,2}).
$$
This completes the proof of the lemma.
\end{proof}

We now turn back to the two terms on the right-hand side of the
error equation (\ref{error}). The first term is given by
$\sum_{T\in\T_h}\langle{\color{black}{(\textbf{q}-\mathbb{Q}_h \textbf{q})}} \cdot \bn, {\S}(v_b)-v_b\rangle_{\partial T}$ which is the central topic of
analysis as shown in the following lemma. Recall that ${\S}(v_b)$ is a linear fitting of $v_b$ on the element $T$ by using the formula (\ref{svb:001}). It is easy to see that the four outward
normal vectors on the element boundary are given as column vectors by
$\bn_1=(-1,0)'$, $\bn_2=(1,0)'$, $\bn_3=(0,-1)'$, and $\bn_4=(0,1)'$.

From Lemma \ref{midvalue}, the function ${\S}(v_b)-v_b$ has the same value at the midpoints of $e_1$ and $e_2$. Furthermore, ${\S}(v_b)-v_b$ has the same directional derivative along
$e_1$ and $e_2$ which is $\partial_y {\S}(v_b)=\nabla_d v_b \cdot \bn_4$. Thus, ${\S}(v_b)-v_b$ has the same value along $e_1$ and $e_2$ at the symmetric points $(x_{i-1}, y)$ and $(x_{i}, y)$. Similarly, ${\S}(v_b)-v_b$ has the same value along $e_3$ and $e_4$ at the symmetric points $(x, y_{j-1})$ and $(x, y_j)$.  It follows that
\begin{equation}\label{EQ:12:29:100}
\langle  {\color{black}{\mathbb{Q}_h \textbf{q}}}\cdot \bn, {\S}(v_b)-v_b\rangle_{\partial T}
=0.
\end{equation}

As ${\S}(v_b)-v_b$ has the same value at the points $(x_{i-1},
y)$ and $(x_{i}, y)$, this boundary function can be extended to the
rectangular element $T$ by assuming the value $({\S}(v_b)-v_b)(x_{i-1},y)$ along each horizontal line segment. For simplicity, we denote this extension as $\chi_1$, i.e.,
\begin{equation}\label{EQ:chi-one}
\chi_1(x,y):=({\S}(v_b)-v_b)(x_{i-1},y),\qquad (x,y)\in T.
\end{equation}
Analogously, ${\S}(v_b)-v_b$ can be extended to $T$ by using its
information on the edges $e_3$ and $e_4$, yielding
\begin{equation}\label{EQ:chi-two}
\chi_2(x,y):=({\S}(v_b)-v_b)(x,y_{j-1}),\qquad (x,y)\in T.
\end{equation}
From Lemma \ref{midvalue}, we have
\begin{equation}\label{EQ:chi-properties}
\begin{split}
& \partial_x \chi_1  = 0,\quad \partial_y \chi_1  = \nabla_d v_b \cdot \bn_4,\\
& \partial_y \chi_2  = 0, \quad \partial_x \chi_2  = \nabla_d v_b \cdot \bn_2,\\
& |e_1|\chi_1(M_1) = - |e_{3}|\chi_2(M_{3}).
\end{split}
\end{equation}

\begin{lemma}\label{Lemma:Lemma-001}
Let ${\color{black}{u}}\in H^3(\Omega)$ be a given function, and $\T_h = \Delta_x \times \Delta_y $ be the rectangular partition. On each element $T\in \T_h$ depicted as in Figure \ref{rectangular-element}, for any $v_{b}\in
V_b^{0}$ we have the following expansion
\begin{equation}\label{est1:010}
\begin{split}
 &\langle {\color{black}{(\textbf{q}-\mathbb{Q}_h \textbf{q})}}\cdot {\color{black}{\bn}}, {\S}(v_b)-v_b\rangle_{\partial T} \\
= & {\color{black}{\chi_1(M_1) \int_T q_{1x} dT+ \chi_2(M_3) \int_T q_{2y} dT}} + R_1(T),
\end{split}
\end{equation}
{\color{black}{where $\textbf{q}=(q_1, q_2){\color{black}{'}}=(a_{11} u_x+a_{12} u_y, a_{21} u_x+a_{22} u_y){\color{black}{'}}$, $q_{1x}=\frac{\partial q_1}{\partial x}$, $q_{2y}=\frac{\partial q_2}{\partial y}$.}}
The remainder term $R_1(T)$ satisfies the following estimate
$$
\sum_{T\in\T_h}|R_{1}(T)|\leq C h^{2} \| \bq\|_{2} \|\nabla_d v_b\|_0.
$$
\end{lemma}

\begin{proof} From (\ref{EQ:12:29:100}) and the structure of $\chi_i$, we have
\begin{equation}\label{est1}
\begin{split}
& \langle  {\color{black}{(\textbf{q}-\mathbb{Q}_h \textbf{q})}}\cdot {\color{black}{\bn}}, {\S}(v_b)-v_b\rangle_{\partial T}\\
= & \ \langle {\color{black}{\textbf{q}}}\cdot \bn, {\S}(v_b)-v_b\rangle_{\partial T}\\
= & \ -\int_{e_1} {\color{black}{q_1}} \chi_1 dy+  \int_{e_2}  {\color{black}{q_1}}  \chi_1 dy - \int_{e_3} {\color{black}{q_2}}  \chi_2 dx \\&+ \int_{e_4}{\color{black}{q_2}}  \chi_2 dx\\
= & {\color{black}{\int_{T} q_{1x} \chi_1 dT+  \int_{T} {q_{2y}  \chi_2} dT  }}.
\end{split}
\end{equation}
Since $\chi_1$ is linear in the $y$-direction and constant in the $x$-direction, then
$$
\chi_1(y) = \chi_1(M_1) + (y-y_c) \partial_y \chi_1 .
$$
Thus, we have
{\color{black}{ \begin{eqnarray*}
 \int_{T} q_{1x} \chi_1 dT &=& \int_{T} q_{1x} \chi_1(M_1) dT +
\int_{T} q_{1x} (y-y_c) \partial_y \chi_1 dT\\
&=& \int_{T} q_{1x} \chi_1(M_1) dT + \int_{T} q_{1xy} E_3(y)
\partial_y \chi_1 dT,
\end{eqnarray*}}}
where $E_3(y) = \frac18 |e_1|^2 - \frac12 (y-y_c)^2$.

Similarly, one may derive the following
{\color{black}{\begin{eqnarray*}
 \int_{T} q_{2y} \chi_2 dT &=& \int_{T} q_{2y}\chi_2(M_3) dT +
\int_{T} q_{2y}(x-x_c) \partial_x \chi_2 dT\\
&=& \int_{T} q_{2y} \chi_2(M_3) dT + \int_{T} q_{2yx} E_4(x)
\partial_x \chi_2 dT,
\end{eqnarray*}}}
where $E_4(x) = \frac18 |e_3|^2 - \frac12 (x-x_c)^2$.

Substituting the last two identities into (\ref{est1}) yields
{\color{black}{
\begin{equation}\label{est1:002}
\begin{split}
& \langle   (\textbf{q}-\mathbb{Q}_h \textbf{q})\cdot {\color{black}{\bn}}, {\S}(v_b)-v_b\rangle_{\partial T}\\
= &  \int_{T} q_{1x} \chi_1(M_1) dT  +\int_{T} q_{2y} \chi_2(M_3) dT \\
&+ \int_{T} q_{1xy} E_3(y)
\partial_y \chi_1 dT+ \int_{T} q_{2yx} E_4(x)
\partial_x \chi_2 dT.
\end{split}
\end{equation}}}
The sum of the last two terms in (\ref{est1:002}) makes the remainder $R_1(T)$ which can be bounded as follows:

{\color{black}{
\begin{eqnarray*}
\begin{split}
&|\int_{T} q_{1xy} E_3(y)
\partial_y \chi_1 dT+ \int_{T} q_{2yx} E_4(x)
\partial_x \chi_2 dT|\\
 \leq&  C
h^2 \|\nabla^2 \bq\|_{T} (\|\partial_x \chi_2\|_T+\|\partial_y \chi_1\|_T) \\
 \leq&  C h^2 \|\nabla^2 \bq\|_{T} \|\nabla_d v_b\|_T,
\end{split}
\end{eqnarray*}
}}
where in the last step we have used the property (\ref{EQ:chi-properties}). This completes the proof of the lemma.
\end{proof}

Next, we shall deal with the second term on the right-hand side of
(\ref{EQ:functional}); namely,
$$
\rho h^{-1} {\color{black}{\sum_{T\in {\cal T}_h}}}\langle Q_b {\S}(Q_bu)-Q_bu, Q_b {\S}(v_b)-v_b\rangle_{\partial T},
$$
where $Q_b u|_{e_i}=\frac{1}{|e_i|} \int_{e_i} u ds$ is the average of $u$ on edge $e_i$.

\begin{lemma}\label{LEMMA:Lemma:002}
Under the assumptions of Lemma \ref{Lemma:Lemma-001}, one has the following expansion
\begin{equation}\label{est2:001}
\begin{split}
&\rho  h^{-1} \langle Q_b{\S}(Q_b{\color{black}{u}})-Q_b{\color{black}{u}}, Q_b{\S}(v_b)-v_b\rangle_{\partial T} \\
= &-{\color{black}{\rho h^{-1} A_1}} \left(|e_3|\chi_1(M_1)\int_T {\color{black}{u_{xx}}} dT + |e_1|
\chi_2(M_3) \int_T {\color{black}{u_{yy}}} dT \right) + R_2(T),
\end{split}
\end{equation}
where $A_1=1/6$ and the remainder term $R_2(T)$ has the following estimate:
\begin{equation}\label{EQ:Jan-14:001}
\sum_{T\in\T_{h}}|R_{2}(T)|\leq C h^{2} \|\nabla^3 u\| \3bar {\S}(v_{b})-v_{b}\3bar.
\end{equation}
Here and in what follows of this paper, we define
\begin{equation}\label{EQ:S-I}
\3bar {\S}(v_{b})-v_{b}\3bar^{2}:=\rho h^{-1}\sum_{T\in\T_h}  \langle Q_{b}{\S}(v_{b})-v_{b},Q_{b}{\S}(v_{b})-v_{b}\rangle_{\partial T}.
\end{equation}
\end{lemma}

\begin{proof} From the equation (\ref{EQ:extension-S:new}) and the properties \eqref{EQ:chi-properties} we have
\begin{equation}\label{est2}
\begin{split}
&\rho  h^{-1} \langle Q_b {\S}(Q_bu)-Q_bu, Q_b {\S}(v_b)-v_b\rangle_{\partial T} \\
=&-\rho  h^{-1} \langle  Q_bu, {\S}(v_b)-v_b\rangle_{\partial T} \\
=&-\rho  h^{-1}\Big( |e_1|  Q_bu(M_1) \chi_1(M_1) +
 |e_2|  Q_bu(M_2) \chi_1(M_2)\\
 &+ |e_3|  Q_bu(M_3) \chi_2(M_3)+ |e_4|  Q_bu(M_4) \chi_2(M_4)\Big)\\
=&-\rho  h^{-1} |e_1|\chi_1(M_1) \Big( Q_bu(M_1) +
 Q_b u(M_2) -  Q_b u(M_3) - Q_b u(M_4) \Big). \\
\end{split}
\end{equation}
Now using the Euler-MacLaurin formula we arrive at
\begin{equation}\label{EQ:12:29:001}
\begin{split}
& \ |e_1|\ |e_3|(Q_bu(M_1) + Q_b u(M_2))\\
 = & \ |e_3| \int_{e_1} u(x_{i-1},y) dy +|e_3|
\int_{e_2} u(x_i,y) dy \\
= & \ 2\int_T u(x,y) dT + A_1 |e_3|^2 \int_T u_{xx} dT + A_2
|e_3|^3 \int_T u_{xxx} E_1(x) dT,\\
\end{split}
\end{equation}
where {\color{black}{$A_2$ is a constant}}, and $E_1$ is a cubic
polynomial in the $x$-direction.

Analogously, we have
\begin{equation}\label{EQ:12:29:002}
\begin{split}
&\ |e_1|\ |e_3|(Q_bu(M_3) + Q_b u(M_4))\\
 = & \ |e_1|\int_{e_3} u(x,y_{j-1}) dx + |e_1|
\int_{e_4} u(x,y_j) dx \\
= & \ 2\int_T u(x,y) dT + A_1 |e_1|^2 \int_T u_{yy} dT + A_2
|e_1|^3 \int_T u_{yyy} E_2(y) dT,\\
\end{split}
\end{equation}
where $E_2$ is a cubic polynomial in the $y$-direction.

Substituting (\ref{EQ:12:29:001}) and (\ref{EQ:12:29:002}) into
(\ref{est2}) yields
\begin{equation*}
\begin{split}
&\rho  h^{-1} \langle Q_b{\S}(Q_bu)-Q_bu, Q_b{\S}(v_b)-v_b\rangle_{\partial T} \\
=& -\rho  h^{-1} |e_1|\chi_1(M_1) \Big(A_1 |e_3|\ |e_1|^{-1}
\int_T u_{xx} dT - A_1 |e_1| |e_3|^{-1} \int_T u_{yy} dT\\
 &  +A_2|e_3|^2 |e_1|^{-1}\int_T u_{xxx}E_1(x) dT -A_2 |e_1|^2
|e_3|^{-1}\int_T u_{yyy}E_2(y) dT\Big) \\
=&-\rho  h^{-1} A_1 \left(|e_3|\chi_1(M_1)\int_T u_{xx} dT + |e_1|
\chi_2(M_3) \int_T u_{yy} dT\right) + R_2(T),
\end{split}
\end{equation*}
where we have used the relation $|e_1|\chi_1(M_1) = -
|e_3|\chi_2(M_3)$, and $R_2(T)$ is given by
\begin{equation}\label{EQ:remainderR5}
\begin{split}
R_2(T)=&-A_2 \rho h^{-1}\left(|e_3|^2 \chi_1(M_1)\int_T u_{xxx}E_1(x)
dT \right.\\
& \left. +|e_1|^2\chi_2(M_3)\int_T u_{yyy}E_2(y) dT\right),
\end{split}
\end{equation}
which can be seen to satisfy (\ref{EQ:Jan-14:001}).
\end{proof}

The following expansion for the error function provides a basis for superconvergence.

\begin{theorem}\label{THM:SupC-expansion}
Assume that $u\in H^3(\Omega)$ is the exact solution of the model problem (\ref{a1})-(\ref{aa1}), and {\color{black}{$u_{b}\in V_{b}^{g}$}} is the weak Galerkin finite element approximation arising from (\ref{WG-scheme}).
On each element $T$ (see Fig. \ref{rectangular-element}), define $w_b\in V_b$ as follows
{\color{black}{\begin{equation}\label{EQ:wb-pre:new}
w_b=\left\{
\begin{array}{l}
\frac{1}{12}|e_1| \rho^{-1}h^{-1} (\rho  h^{-1}|e_1| Q_b(u_{yy})|_{e_1}-6 Q_b(q_{2y})
 |_{e_1} ),\quad \mbox{on } e_1,\\
\frac{1}{12}|e_2| \rho^{-1}h^{-1} (\rho  h^{-1}|e_2| Q_b(u_{yy})|_{e_2}-6 Q_b(q_{2y})
 |_{e_2} ),\quad \mbox{on } e_2,\\
\frac{1}{12}|e_3| \rho^{-1}h^{-1} ( \rho  h^{-1} |e_3|Q_b(u_{xx})|_{e_3}-6  Q_b
(q_{1x})|_{e_3}),\quad \mbox{on } e_3,\\
\frac{1}{12} |e_4| \rho^{-1}h^{-1}( \rho  h^{-1} |e_4|Q_b(u_{xx})|_{e_4}-6  Q_b
(q_{1x})|_{e_4}),\quad \mbox{on } e_4.
\end{array}
\right.
\end{equation}
}}
Denote by $\widetilde{e}_b=(Q_b u - {\color{black}{u_b}})+ h^2 w_b$ the modified error function. Then, the following equation or expansion holds true:
\begin{equation}\label{EQ:marker:new}
%\begin{split}
({\color{black}{a}}\nabla_d \tilde{e}_b, \nabla_d v_b)+
s(\tilde{e}_b, v_b)= h^2 ({\color{black}{a}}\nabla_d w_b, \nabla_d v_b)+R_3(v_b)
%\end{split}
\end{equation}
for all $v_b\in V_b^0$, where $R_3(v_b)$ is the remainder satisfying
\begin{equation}\label{EQ:remainder}
|R_{3}(v_b)|\leq C h^{2}\| u\|_{3} \ \3bar {\S}(v_{b})-v_{b}\3bar.
\end{equation}
\end{theorem}

\begin{proof}
First of all, from (\ref{EQ:wb-pre:new}) it is easy to see that the function $w_b$ is single-valued on each edge. Furthermore, the following estimate holds true
\begin{equation}\label{EQ:wb-estimate}
\|\nabla_d w_b\| \leq C \|u\|_3.
\end{equation}

The proof of \eqref{EQ:marker:new} is merely a combination of the error equation
(\ref{error}) with the two expansions given in Lemmas
\ref{Lemma:Lemma-001} and \ref{LEMMA:Lemma:002}. To this end,
note that the functional $\zeta_u$ on the right-hand side of
(\ref{error}) consists of two terms detailed in
(\ref{EQ:functional}). The first term of $\zeta_u(v_b)$ has the
expansion (\ref{est1:010}) and the second one has (\ref{est2:001}) on each element $T$.
They collectively give the following expansion ($A_1=1/6$)
{\color{black}{
\begin{equation}\label{EQ:Jan-03:001}
\begin{split}
\zeta_u(v_b) = &\sum_{T\in\T_h}
-\rho  h^{-1} A_1
|e_3| \chi_1(M_1)\int_T u_{xx}
dT \\
& - \sum_{T\in\T_h} \rho  h^{-1} A_1 |e_1| \chi_2(M_3) \int_T u_{yy} dT\\
& +\sum_{T\in\T_h} \chi_1(M_1)\int_T q_{1x}dT+\sum_{T\in\T_h} \chi_2(M_3)\int_T q_{2y}dT \\
& + \sum_{T\in\T_h} (R_1(T)+R_2(T)).
\end{split}
\end{equation}
}}
On the rectangular element $T=[x_{i-1}, x_i]\times[y_{j-1}, y_j]$,
note that
$$
\int_T u_{xxy} (y-y_c) dT = - \int_T u_{xx} dT + \frac12 |e_1|
\int_{e_4} u_{xx} dx + \frac12 |e_1| \int_{e_3} u_{xx} dx.
$$
Thus,
$$
\int_T u_{xx} dT =  \frac12 |e_1| \ |e_4| \ Q_b (u_{xx})|_{e_4} +
\frac12 |e_1|\ |e_3| \ Q_b (u_{xx})|_{e_3}- \int_T u_{xxy} (y-y_c)
dT.
$$
Analogously, we have
$$
\int_T u_{yy} dT =  \frac12 |e_3| \ |e_2| \ Q_b (u_{yy})|_{e_2} +
\frac12 |e_3|\ |e_1| \ Q_b (u_{yy})|_{e_1}- \int_T u_{yyx} (x-x_c)
dT.
$$
$$
{\color{black}{\int_T q_{1x} dT =  \frac12 |e_1| \ |e_4| \ Q_b (q_{1x})|_{e_4} +
\frac12 {\color{black}{|e_1|}}\ |e_3| \ Q_b (q_{1x})|_{e_3}- \int_T q_{1xy} (y-y_c)
dT.
}}
$$
$$
{\color{black}{\int_T q_{2y}dT =  \frac12 |e_1| \ {\color{black}{|e_3|}} \ Q_b ( q_{2y})|_{e_1} +
\frac12 |e_3|\ |e_2| \ Q_b ( q_{2y})|_{e_2}- \int_T  q_{2yx} (x-x_c)
dT.
}}
$$

Substituting the last four identities into (\ref{EQ:Jan-03:001})
yields
{\color{black}{
\begin{equation}\label{EQ:Jan-03:002}
\begin{split}
\zeta_u(v_b) = &-\frac12 \sum_{T\in\T_h} \rho  h^{-1} A_1
|e_1| \ |e_3|^2 \chi_1(M_1)   (Q_b (u_{xx})|_{e_3} + Q_b
(u_{xx})|_{e_4})\\
& - \frac12 \sum_{T\in\T_h} \rho  h^{-1} A_1 |e_1|^2 \ |e_3| \chi_2(M_3)   (Q_b (u_{yy})|_{e_1} + Q_b (u_{yy})|_{e_2})\\
& +\frac12 \sum_{T\in\T_h} \chi_1(M_1)|e_1|\ |e_3| (Q_b (q_{1x})|_{e_3} + Q_b
(q_{1x})|_{e_4})\\
& +\frac12 \sum_{T\in\T_h} \chi_2(M_3)|e_3|\ |e_1| (Q_b (q_{2y})|_{e_1} + Q_b
(q_{2y})|_{e_2})+ \sum_{T\in\T_h} R_3(T),
\end{split}
\end{equation}}}
where
{\color{black}{
\begin{equation*}
\begin{split}
R_3(T) = & R_1(T)+R_2(T) \\
& + \rho  h^{-1} A_1 |e_3| \chi_1(M_1)\int_T u_{xxy} (y-y_c) dT\\
& + \rho  h^{-1} A_1 |e_1|  \chi_2(M_3)\int_T u_{yyx} (x-x_c) dT\\
& - \chi_1(M_1)\int_T q_{1xy} (y-y_c) dT- \chi_2(M_3)\int_T q_{2yx} (x-x_c) dT
\end{split}
\end{equation*}
}}
is the combined remainder term on the element $T$. It is not hard to
see that the combined remainder term can be bounded as follows
$$
\sum_{T\in\T_{h}}|R_{3}(T)|\leq C h^{2}\| u\|_{3} \ \3bar {\S}(v_{b})-v_{b}\3bar.
$$
Thus, it suffices to deal with the leading term in the
expansion (\ref{EQ:Jan-03:002}). To this end, we use the properties
(\ref{EQ:chi-properties}) to rewrite (\ref{EQ:Jan-03:002}) as
follows
{\color{black}{
\begin{equation}\label{EQ:Jan-03:003}
\begin{split}
&\zeta_u(v_b) \\
= &\frac12 \sum_{T\in\T_h} \rho  h^{-1} A_1
  |e_3|^3 (\chi_2(M_3)\ Q_b (u_{xx})|_{e_3} +\chi_2(M_4) Q_b
(u_{xx})|_{e_4})\\
& + \frac12 \sum_{T\in\T_h}  \rho  h^{-1} A_1  |e_1|^3 (\chi_1(M_1)\ Q_b (u_{yy})|_{e_1} + \chi_1(M_2)\
Q_b (u_{yy})|_{e_2})\\
& -\frac12 \sum_{T\in\T_h} |e_3|^2(\chi_2(M_3)\ Q_b (q_{1x})|_{e_3} +\chi_2(M_4) Q_b
(q_{1x})|_{e_4}                    )  \\
& -\frac12 \sum_{T\in\T_h} |e_1|^2(\chi_1(M_1)\ Q_b (q_{2y})|_{e_1} +\chi_1(M_2) Q_b
(q_{2y})|_{e_2}                    )  \\
& + \sum_{T\in\T_h} R_3(T).
\end{split}
\end{equation}}}
By introducing
{\color{black}{\begin{equation}\label{EQ:wb}
w_b=\left\{
\begin{array}{l}
\frac{1}{2}|e_1| \rho^{-1}h^{-1} ( \rho  h^{-1}A_1|e_1|  Q_b
(u_{yy})|_{e_1} - Q_b
(q_{2y})|_{e_1}  ),\quad \mbox{on } e_1,\\
\frac{1}{2}|e_2| \rho^{-1}h^{-1} ( \rho  h^{-1}A_1 |e_2| Q_b
(u_{yy})|_{e_2} - Q_b
(q_{2y})|_{e_2}  ),\quad \mbox{on } e_2,\\
\frac{1}{2}|e_3| \rho^{-1}h^{-1} ( \rho  h^{-1}A_1 |e_3|  Q_b
(u_{xx})|_{e_3} - Q_b
(q_{1x})|_{e_3}  ),\quad \mbox{on } e_3,\\
\frac{1}{2} |e_4| \rho^{-1}h^{-1}( \rho  h^{-1}A_1 |e_4|  Q_b
(u_{xx})|_{e_4} - Q_b
(q_{1x})|_{e_4}  ),\quad \mbox{on } e_4.
\end{array}
\right.
\end{equation}}}
we may rewrite (\ref{EQ:Jan-03:003}) in the following form
\begin{equation}\label{EQ:Jan-03:004}
\zeta_u(v_b) = \rho h \sum_{T\in\T_h} \langle w_b, Q_b{\S}(v_b)-v_b \rangle_{\partial T} + {\color{black}{\sum_{T\in\T_h}}}R_3(T).
\end{equation}
Note that the weak function $w_b$ is well-defined by (\ref{EQ:wb})
as the value on each interior edge is uniquely determined by this
formula. Thus, with $e_b= Q_bu-u_{b}$, we have
\begin{equation*}
\begin{split}
&\sum_{T\in {\cal T}_h}({\color{black}{a}}\nabla_d e_b, \nabla_d v_b)_T+\rho h^{-1}{\color{black}{\sum_{T\in\T_h}}}\langle Q_b{\S}(e_b)-e_b, Q_b{\S}(v_b)-v_b \rangle_{\partial T} \\
=&\rho h \sum_{T\in {\cal T}_h} \langle w_b, Q_b {\S}(v_b)-v_b \rangle_{\partial T} + {\color{black}{\sum_{T\in\T_h}}}R_3(T)\\
=&-\rho h \sum_{T\in {\cal T}_h} \langle Q_b {\S}(w_b)-w_b,
{\color{black}{Q_b}}{\S}(v_b)-v_b \rangle_{\partial T} + {\color{black}{\sum_{T\in\T_h}}}R_3(T).
\end{split}
\end{equation*}
By letting $\tilde{e}_b=e_b+ h^2 w_b$ we arrive at
\begin{equation}\label{EQ:marker}
\begin{split}
&\sum_{T\in {\cal T}_h}({\color{black}{a}}\nabla_d \tilde{e}_b, \nabla_d v_b)_T+\rho
h^{-1} {\color{black}{\sum_{T\in\T_h}}}\langle Q_b{\S}(\tilde{e}_b)-\tilde{e}_b, Q_b{\S}(v_b)-v_b \rangle_{\partial T} \\
= &\sum_{T\in {\cal T}_h} h^2 ({\color{black}{a}}\nabla_d w_b, \nabla_d v_b)_T+{\color{black}{\sum_{T\in\T_h}}}R_3(T),
\end{split}
\end{equation}
which gives precisely the expansion \eqref{EQ:marker:new} with $R_3(v_b)=\sum_{T\in\T_h}R_3(T)$
\end{proof}

\subsection{Superconvergence}

The error expansion \eqref{EQ:marker:new} in Theorem \ref{THM:SupC-expansion} indicates that the modified error function $\widetilde e_b =(Q_b u - {\color{black}{u_b}})+ h^2 w_b$ satisfies an equation with load function at the scale of $\OO(h^2)$. If $\widetilde e_b$ were to be vanishing on the boundary $\partial\Omega$, then one would obtain an estimate of the following type
$$
(a \nabla_{d}\widetilde{e_{b}}, \nabla_{d}\widetilde{e_{b}})^{1/2} \leq C h^2 \|\nabla^3 u\|
$$
by letting $v_b=\widetilde{e_{b}}$ in \eqref{EQ:marker:new}. The question is when it would be possible to have $\widetilde e_b|_{\partial\Omega}=0$. From $\widetilde e_b =(Q_b u - u_b)+ h^2 w_b$ we see that the only way to have $\widetilde e_b|_{\partial\Omega}=0$ is to enforce a computational solution $u_b$ satisfying the following boundary condition
\begin{equation}\label{EQ:comp-bdry-cond}
u_b|_{\partial\Omega} = Q_b g + h^2 w_b|_{\partial\Omega}.
\end{equation}
The above boundary condition can be implemented if $w_b|_{\partial\Omega}$ is computable without any prior knowledge of the exact solution $u$. The following result assumes a computable $w_b|_{\partial\Omega}$.

\begin{theorem}\label{THM:SuperC-nonhomo}
Assume that $u\in H^3(\Omega)$ is the exact solution of the model problem (\ref{a1})-(\ref{aa1}). Let $w_b\in V_b$ be given on each element $T$ (see Fig. \ref{rectangular-element}) as follows
{\color{black}{\begin{equation}\label{EQ:wb-pre-new}
w_b=\left\{
\begin{array}{l}
\frac{1}{12}|e_1| \rho^{-1}h^{-1} (\rho  h^{-1}|e_1| Q_b(u_{yy})|_{e_1}-6 Q_b(q_{2y})
 |_{e_1} ),\quad \mbox{on } e_1,\\
\frac{1}{12}|e_2| \rho^{-1}h^{-1} (\rho  h^{-1}|e_2| Q_b(u_{yy})|_{e_2}-6 Q_b(q_{2y})
 |_{e_2} ),\quad \mbox{on } e_2,\\
\frac{1}{12}|e_3| \rho^{-1}h^{-1} ( \rho  h^{-1} |e_3|Q_b(u_{xx})|_{e_3}-6  Q_b
(q_{1x})|_{e_3}),\quad \mbox{on } e_3,\\
\frac{1}{12} |e_4| \rho^{-1}h^{-1}( \rho  h^{-1} |e_4|Q_b(u_{xx})|_{e_4}-6  Q_b
(q_{1x})|_{e_4}),\quad \mbox{on } e_4.
\end{array}
\right.
\end{equation}
}}
Let {\color{black}{$u_{b}\in V_{b}^g$}} be the weak Galerkin finite element approximation arising from (\ref{WG-scheme}) with the following boundary value
\begin{equation}\label{EQ:bdry-value}
u_b|_{\partial\Omega} = \widetilde{Q}_b g:=Q_b g + h^2 w_b|_{\partial\Omega}.
\end{equation}
Denote by $\widetilde{e}_b=(Q_b u + h^2 w_b) - {\color{black}{u_b}}$ the modified error function. Then, the following estimate holds true:
\begin{equation}\label{EQ:SuperC:008}
\left(\sum_{T\in\T_{h}}\|\nabla_{d}\widetilde{e_{b}} \|_{T}^{2}\right)^{\frac12}+ \3bar {\S}(\widetilde{e_{b}})-\widetilde{e_{b}} \3bar \leq Ch^{2}\|u\|_{3}.
\end{equation}
\end{theorem}

\begin{proof} From the error expansion \eqref{EQ:marker:new} in Theorem \ref{THM:SupC-expansion}, we have
\begin{equation}\label{EQ:marker:new:001}
(a\nabla_d \widetilde{e}_b, \nabla_d v_b)+
s(\widetilde{e}_b, v_b)= h^2 (\nabla_d w_b, \nabla_d v_b)+R_3(v_b)
\end{equation}
for all $v_b\in V_b^0$, where the remainder $R_3(v_b)$ has the estimate \eqref{EQ:remainder}. As $u_b|_{\partial\Omega}= \widetilde{Q}_b g$, it follows from \eqref{EQ:bdry-value} that $\widetilde{e_{b}}|_{\partial\Omega} =0$ so that $\widetilde{e_{b}}\in V_b^0$. By letting $v_b= \widetilde{e_{b}}$ in \eqref{EQ:marker:new:001} we obtain
\begin{equation}\label{EQ:marker:new:002}
(a\nabla_d \widetilde{e}_b, \nabla_d \widetilde{e}_b)+
s(\widetilde{e}_b, \widetilde{e}_b)= h^2 (\nabla_d w_b, \nabla_d \widetilde{e}_b)+R_3(\widetilde{e}_b),
\end{equation}
which, together with \eqref{EQ:remainder} and \eqref{EQ:wb-estimate}, yields the superconvergence estimate \eqref{EQ:SuperC:008}.
\end{proof}

From (\ref{EQ:bdry-value}), we see that the usual $L^2$ projection of the Dirichlet data was perturbed by
\begin{equation}\label{EQ:pert:001}
\varepsilon_b:=\frac{1}{12}|e_1|((|e_1|-6\rho^{-1}h{\color{black}{a_{22}}}) Q_b (g_{yy})-{\color{black}{6\rho^{-1}h a_{21}Q_b (u_{yx})}})
\end{equation}
on vertical segments, and by
\begin{equation}\label{EQ:pert:002}
\varepsilon_b:=
\frac{1}{12}|e_3|((|e_3|-6\rho^{-1}h{\color{black}{a_{11}}}) Q_b(g_{xx})-{\color{black}{6\rho^{-1}h a_{12}Q_b (u_{xy})}})
\end{equation}
on horizontal segments. For Dirichlet boundary value problem with diagonal diffusive coefficient $a=(a_{11}, 0; 0, a_{22})$, the perturbation $\varepsilon_b$ is computable by using merely the boundary data $g$, as $a_{12}=a_{21}=0$ in \eqref{EQ:pert:001}-\eqref{EQ:pert:002} so that the mixed partial derivative $u_{xy}$ is not needed. Consequently, the superconvergence estimate \eqref{EQ:SuperC:008} is applicable to Dirichlet boundary value problems with diagonal diffusive tensor.

From Theorem \ref{THM:SuperC-nonhomo}, we have the following estimate
\begin{equation*}\label{EQ:Jan22:001}
\left(\sum_{T\in\T_{h}}\|\nabla_{d}\widetilde{e_{b}} \|_{T}^{2}\right)^{\frac12} \leq C h^2 \|u\|_3,
\end{equation*}
where $\widetilde{e_{b}} = (Q_bu + h^2w_b) - {\color{black}{u_b}}$. It follows that
\begin{equation}\label{EQ:Jan22:002}
\left(\sum_{T\in\T_{h}}\|\nabla_{d} (Q_bu + h^2w_b) - \nabla_d u_b \|_{T}^{2}\right)^{\frac12} \leq C h^2 \|u\|_3.
\end{equation}
By using (\ref{EQ:q0}) and (\ref{EQ:wb-estimate}) in (\ref{EQ:Jan-03:002}) we arrive at the following superconvergence for $\nabla u$:
\begin{equation}\label{EQ:Jan22:003}
\left(\sum_{T\in\T_{h}}\|\mathbb{Q}_h (\nabla u) - \nabla_d u_b \|_{T}^{2}\right)^{\frac12} \leq C h^2 \|u\|_3.
\end{equation}
The result can be summarized as follows.

\begin{corollary}\label{Corollary:SuperC-nonhomo}
Assume that $u\in H^3(\Omega)$ is the exact solution of the model problem (\ref{a1})-(\ref{aa1}) with diagonal diffusive tensor $a=(a_{11}, 0; 0, a_{22})$. Let {\color{black}{$u_{b}\in V_{b}$}} be the weak Galerkin finite element approximation arising from (\ref{WG-scheme}) with the following boundary value:
\begin{equation}\label{EQ:PL2:01}
u_h=Q_b(g)+\frac{1}{12}h_y(h_y-6\rho^{-1}a_{22}h) Q_b (g_{yy})
\end{equation}
on vertical segments, and
\begin{equation}\label{EQ:PL2:02}
u_h=Q_b(g)+\frac{1}{12}h_x(h_x-6\rho^{-1}a_{11}h) Q_b(g_{xx})
\end{equation}
on horizontal segments, where $h_y=|e_1|$ is the meshsize in $y$-direction and $h_x=|e_3|$ is the one in $x$-direction. Then, the following error estimate holds true:
\begin{equation}\label{EQ:SuperConv}
\left(\sum_{T\in\T_{h}}\| \mathbb{Q}_h (\nabla u) - \nabla_d u_b\|_{T}^{2}\right)^{\frac12}\leq Ch^{2}\| u\|_{3}.
\end{equation}
\end{corollary}

For model problems with arbitrary diffusive coefficients, we have the following superconvergence.

\begin{theorem}\label{THM:SuperC-generic}
Assume that $u\in H^3(\Omega)$ is the exact solution of the model problem (\ref{a1})-(\ref{aa1}). Let {\color{black}{$u_{b}\in V_{b}$}} be the weak Galerkin finite element approximation arising from (\ref{WG-scheme}) with the boundary value
\begin{equation}\label{EQ:bdry-value-generic}
\widetilde{Q}_b g := Q_b g.
\end{equation}
Denote by ${e}_b=Q_b u - {\color{black}{u_b}}$ the error function. Then, the following estimate holds true:
\begin{equation}\label{EQ:SuperC:008-generic}
\left(\sum_{T\in\T_{h}}\|\nabla_{d}{e_{b}} \|_{T}^{2}\right)^{\frac12} \leq Ch^{1.5}(\|u\|_{3}+\|\nabla^2 u\|_{0,\partial\Omega}).
\end{equation}
\end{theorem}

\begin{proof} From the error expansion \eqref{EQ:marker:new} of Theorem \ref{THM:SupC-expansion}, we have
\begin{equation}\label{EQ:marker:new:001-generic}
(a\nabla_d \widetilde{e}_b, \nabla_d v_b)+
s(\widetilde{e}_b, v_b)= h^2 (a\nabla_d w_b, \nabla_d v_b)+R_3(v_b)
\end{equation}
for all $v_b\in V_b^0$, where $\widetilde{e}_b=(Q_b u -u_b) + h^2 w_b$ and $w_b$ is given by \eqref{EQ:wb-pre-new}. The remainder $R_3(v_b)$ has the estimate \eqref{EQ:remainder}. The modified error function $\widetilde{e}_b$ is generally non-vanishing on the boundary of the domain so that it is disqualified to serve as a test function. To overcome this difficulty, we shall remove the perturbation $h^2w_b$ from $\widetilde{e}_b$ on the boundary by subtracting the following function
\begin{equation}\label{EQ:wb-boundary}
  \chi_b =
  \left\{
  \begin{array}{rl}
  h^2 w_b &\qquad \mbox{ on edge $e\subset\partial\Omega$},\\
  0  &\qquad \mbox{ otherwise}.
  \end{array}
  \right.
\end{equation}
It then follows from \eqref{EQ:marker:new:001-generic} that
\begin{equation}\label{EQ:marker:new:001-generic:002}
\begin{split}
 &\ (a\nabla_d (\widetilde{e}_b-\chi_b), \nabla_d v_b)+
s(\widetilde{e}_b-\chi_b, v_b)\\
 = & \ h^2 (a\nabla_d w_b, \nabla_d v_b)+R_3(v_b)
 - (a\nabla_d \chi_b, \nabla_d v_b)-s(\chi_b, v_b)
\end{split}
\end{equation}
for all $v_b\in V_b^0$. The first two terms on the right-hand side of \eqref{EQ:marker:new:001-generic:002} can be bounded as follows
\begin{equation}\label{EQ:RHT:000}
  |h^2(a\nabla_d w_b, \nabla_d v_b) + R_3(v_b)| \leq C h^2 \|\nabla^3 u\| (\|\nabla_d v_b\| + \3bar \S(v_b)-v_b \3bar ).
\end{equation}
The third and the fourth term on the right-hand side of \eqref{EQ:marker:new:001-generic:002} can be bounded by using the Schwartz inequality
\begin{equation}\label{EQ:RHT:001}
\begin{split}
  |(a\nabla_d \chi_b, \nabla_d v_b)| &\ \leq C \|\nabla \chi_b\| \|\nabla_d v_b\| \\
  &\ \leq C h^{1.5} \|\nabla^2 u\|_{0,\partial\Omega} \|\nabla_d v_b\|
\end{split}
\end{equation}
and
\begin{equation}\label{EQ:RHT:002}
\begin{split}
  |s(\chi_b, v_b)| &\ \leq  s(\chi_b, \chi_b)^{1/2} s(v_b,v_b)^{1/2} \\
  &\ \leq C h^{1.5} \|\nabla^2 u\|_{0,\partial\Omega} s(v_b,v_b)^{1/2}.
\end{split}
\end{equation}
Substituting the last three estimates into \eqref{EQ:marker:new:001-generic:002} with $v_b= \widetilde{e}_b-\chi_b$ yields the following estimate:
$$
\|\nabla_d (\widetilde{e}_b-\chi_b)\| \leq C (h^2 \|\nabla^3 u\| + h^{1.5} \|\nabla^2 u\|_{0,\partial\Omega}),
$$
where we have also used the Cauchy-Schwarz inequality. The last inequality leads to the desired superconvergence estimate \eqref{EQ:SuperC:008-generic}. This completes the proof of the theorem.
\end{proof}

\section{Numerical Experiments}\label{Section:NE}

In this section, we report some computational results for the weak Galerkin finite element scheme (\ref{WG-scheme}) to numerically justify the superconvergence estimates established in Theorems \ref{THM:SuperC-nonhomo} and \ref{THM:SuperC-generic} and Corollary \ref{Corollary:SuperC-nonhomo}. Our numerical experiment makes use of the lowest order of finite element in the scheme (\ref{WG-scheme}) so that the numerical solution $u_{h}=\{u_{0},u_{b}\}$ is given by $u_{0}\in P_{1}(T)$, $u_{b}\in P_{0}(\partial T)$, and $\nabla_{d}u_{h}\in [P_{0}(T)]^{2}$.

Let $u=u(x,y)$ be the exact solution of (\ref{a1})-(\ref{aa1}), and denote by
$$
e_{h}:= Q_{h}u - u_h=\{e_{0},e_{b}\}
$$
the error function, where $e_{0}= Q_{0}u-{\color{black}{{\S}(u_{b})}}$, $e_{b}= Q_{b}u-u_b$. $Q_{0}u$ and $Q_{b}u$ are the $L^{2}$ projections of the exact solution onto the corresponding finite element spaces. Recall that the extension operator ${\S}$ maps piecewise constant functions on $\pT$ to linear functions on $T$ through the least-squares fitting formula (\ref{EQ:extension-S:new}).

The following metrics are used to measure the error $e_h$ in our numerical experiments:
$$
\displaystyle
\begin{array}{lll}
&\mbox{$L^{2}$-norm:} & \ \displaystyle\|u-{\S}(u_{b})\|_0:= \left(\int_\Omega |u-{\S}(u_{b})|^2d\Omega\right)^{1/2},              \\
&\mbox{Discrete $L^\infty$:} & \ \displaystyle \|u-{\S}(u_{b})\|_{{\color{black}{\infty,\star}}}:=\max_{T\in \T_h} |u(x_c,y_c)-{\S}(u_{b})(x_c,y_c)|,\\
&\mbox{Discrete $H^{1}$:} &\ \displaystyle\|\nabla_d u_b-\nabla u\|_{0,\star}:=\left(\sum_{T\in{\T}_{h}}|\nabla_d u_b-\nabla u(x_c,y_c)|^{2}|T|\right)^{1/2}.
\end{array}
$$
Here $|T|$ denotes the area of the element $T$ and $(x_c,y_c)$ represents the coordinates of the element center.

We consider eleven test examples in our numerical experiments; each addresses a particular feature of the superconvergence theory. The domain for all the test cases is chosen as the square domain with uniform or nonuniform partitions consisting of either squares or rectangles. From Corollary \ref{Corollary:SuperC-nonhomo}, the superconvergence estimate (\ref{EQ:SuperConv}) is possible when the Dirichlet boundary value is approximated by a slightly modified $L^2$ projection of the exact boundary value. The numerical experiment will address the following questions for the superconvergence estimate (\ref{EQ:SuperConv}):
\begin{itemize}
\item Is it necessary to use the modified $L^2$ projection \eqref{EQ:PL2:01}-\eqref{EQ:PL2:02} for the Dirichlet boundary value in the scheme (\ref{WG-scheme})?
    \item Does one has any superconvergence for the numerical solution of (\ref{WG-scheme}) when the $L^2$ projection or the usual nodal point interpolation of the Dirichlet boundary value is employed? If yes, what the rate of superconvergence would be?
\end{itemize}

\subsection{Numerical experiments with constant coefficients}\label{subSection:NE1} We first consider several test examples for the model problem \eqref{a1}-\eqref{aa1} with constant diffusive tensor on the unit square domain.

{\bf Test Case 1 (Homogeneous {BVP}):} The model problem (\ref{a1})-(\ref{aa1}) is defined on the unit square domain $\Omega=(0,1)^2$ with diffusive coefficient tensor given by $a_{11}=a_{22}=1, a_{12}=a_{21}=0$. The exact solution is chosen as $u =\sin(\pi x)\sin(\pi y)$. The solution has vanishing boundary value so that the superconvergence estimate (\ref{EQ:SuperConv}) holds true when the numerical boundary value is set to be zero.

Tables \ref{Example1:rho6:01}-\ref{Example1:rho1:02} illustrate the numerical results for the lowest order WG-FEM on uniform square or rectangular partitions with the stabilization parameter $\rho=6$ and $\rho=1$, respectively. The boundary value was set to be zero in the numerical scheme (\ref{WG-scheme}). As the Dirichlet data is homogeneous, all the assumptions of Corollary \ref{Corollary:SuperC-nonhomo} are satisfied for this test example so that a superconvergence of order ${\cal{O}}(h^2)$ is expected for the gradient approximation. The last two columns of Tables \ref{Example1:rho6:01}-\ref{Example1:rho1:02} show the numerical performance of the weak Galerkin finite element scheme (\ref{WG-scheme}) in various $H^1$ norms. The numerical results clearly confirm the superconvergence theory developed in the previous section.

It is interesting to note that the numerical solutions are very close to each other for the stabilization parameter $\rho=1$ and {$\rho=6$} as shown in Tables \ref{Example1:rho6:01} and \ref{Example1:rho1:01}. We also computed the solution for several other values of $\rho$ (e.g., $\rho = 0.01, 0.1, 2, 5$), and the numerical results stay unchanged in terms of $\rho$ on uniform square partitions. In addition, the two tables \ref{Example1:rho6:01} and \ref{Example1:rho1:01} show a superconvergence of order $4$ in the $H^1$ {norm} - a superconvergence phenomena better than what the theory predicted. We believe this is a special property of the testing example and the result is not generalizable to other problems.

\begin{table}[htbp]
\caption{Test Case 1: Numerical performance of the WG scheme \eqref{WG-scheme} (domain $\Omega=(0,1)^2$, exact solution $u=\sin(\pi x)\sin(\pi y)$, uniform square partitions, stabilization parameter $\rho=6$, and vanishing Dirichlet boundary data{\color{black}{).}}}\label{Example1:rho6:01}
{
\begin{center}\centering\scriptsize
\setlength{\extrarowheight}{1.5pt}
\begin{tabular}{|l|l|l|l|l|l|}
\hline $h$ &$\|u-\S(u_{b})\|_{\infty,\star}$& $\|u-\S(u_{b})\|_0$ & $\|\nabla_d({e_{b}})\|_0$  &$\|\nabla_d u_b-\nabla u\|_{0,\star}$  &    $\|\nabla(Q_{0}u-\S(u_{b}))\|_0$  \\
\hline
1/4     &4.5171e-02    &3.0366e-02     &1.0957e-01                &2.2968e-03    &8.7561e-02 \\    \hline
1/8     &1.2456e-02    &7.6006e-03    &2.8256e-02               &1.4594e-04     &2.2598e-02\\     \hline
1/16   &3.1880e-03    &1.9006e-03     &7.1186e-03                &9.1591e-06    &5.6945e-03 \\      \hline
1/32    &8.01643e-04    &4.7517e-04     &1.7831e-03                &5.7307e-07   &1.4265e-03   \\       \hline
1/64    &2.0070e-04    &1.1879e-04    &4.4599e-04               &3.5824e-08     &3.5679e-04 \\           \hline
1/128   &5.0193e-05   &2.9698e-05     &1.1151e-04             &2.2391e-09     &8.9208e-05\\           \hline
1/256    &1.2549e-05   &7.4246e-06   &2.7878e-05                &1.3994e-010    &2.2303e-05 \\ \hline
\hline
Rate        &2.00          & 2.00      &2.00 & 4.00  &2.00 \\
\hline
\end{tabular}
\end{center}
}
\end{table}
%\end{small}

\begin{table}[htbp]
\caption{Test Case 1: Numerical performance of the WG scheme \eqref{WG-scheme} (domain $\Omega=(0,1)^2$, exact solution $u=\sin(\pi x)\sin(\pi y)$, uniform rectangular partitions, stabilization parameter $\rho=6$, mesh parameter $h=(h_x+h_y)/2$, and vanishing Dirichlet boundary data{\color{black}{).}}}\label{Example1:rho6:02}
{
\setlength{\extrarowheight}{1.5pt}
\begin{center}\centering\scriptsize
\begin{tabular}{|l|l|l|l|l|l|}
\hline $h$ &$\|u-\S(u_{b})\|_{\infty,\star}$& $\|u-\S(u_{b})\|_0$ & $\|\nabla_d({e_{b}})\|_0$  &$\|\nabla_d u_b-\nabla u\|_{0,\star}$  &    $\|\nabla(Q_{0}u-\S(u_{b}))\|_0$  \\
\hline
2.08e-01     &3.2707e-02    &2.1324e-02     &8.2340e-02
&1.8088e-02   &6.5326e-02   \\    \hline
1.04e-01     &8.7358e-03    &5.3199e-03    &2.0891e-02               &4.1037e-03    &1.6580e-02 \\     \hline
5.21e-02    &2.2193e-03    &1.3294e-03     &5.2424e-03                &9.9990e-04    &4.1609e-03 \\      \hline
2.60e-02    &5.5703e-04    &3.3232e-04     &1.3118e-03                &2.4835e-04    &1.0412e-03  \\       \hline
1.30e-02    &1.3940e-04    &8.3079e-05    &3.2803e-04               &6.1987e-05     &2.6037e-04      \\           \hline
6.51e-03   &3.4858e-05   &2.0770e-05   &8.2013e-05                 &1.5490e-05  &6.5096e-05  \\           \hline
3.26e-03    &8.7150e-06   &5.1924e-06   &2.0503e-05         &3.8722e-06    &1.6274e-05\\           \hline
\hline
Rate        &2.00  & 2.00 &2.00  & 2.00  & 2.00  \\
\hline
\end{tabular}
\end{center}
}
\end{table}

Tables \ref{Example1:rho6:02} and \ref{Example1:rho1:02} show the numerical results when rectangular partitions are used in the numerical scheme (\ref{WG-scheme}). The finite element partitions are obtained from an initial $2\times 3$ uniform partition through the usual successive refinement technique; namely, by dividing each rectangle into four equal-sized sub-rectangles. We use $h_x$ to represent the meshsize in $x$-direction and $h_y$ in $y$-direction.

\begin{table}[htbp]
\caption{Test Case 1:
Numerical performance of the WG scheme \eqref{WG-scheme} (domain $\Omega=(0,1)^2$, exact solution $u=\sin(\pi x)\sin(\pi y)$, uniform square partitions, stabilization parameter $\rho=1$, and vanishing Dirichlet boundary data{\color{black}{).}}}\label{Example1:rho1:01}
{
\setlength{\extrarowheight}{1.5pt}
\begin{center}\centering\scriptsize
\begin{tabular}{|l|l|l|l|l|l|}
\hline $h$ &$\|u-\S(u_{b})\|_{\infty,\star}$& $\|u-\S(u_{b})\|_0$ & $\|\nabla_d({e_{b}})\|_0$  &$\|\nabla_d u_b-\nabla u\|_{0,\star}$  &    $\|\nabla(Q_{0}u-\S(u_{b}))\|_0$  \\
\hline
\hline
1/4     &4.5171e-02    &3.0366e-02     &1.0957e-01                &2.2968e-03    &8.7561e-02 \\    \hline
1/8     &1.2456e-02    &7.6006e-03    &2.8256e-02               &1.4594e-04     &2.2598e-02\\     \hline
1/16   &3.1880e-03    &1.9006e-03     &7.1186e-03                &9.1591e-06    &5.6945e-03 \\      \hline
1/32    &8.01643e-04    &4.7517e-04     &1.7831e-03                &5.7307e-07   &1.4265e-03   \\       \hline
1/64    &2.0070e-04    &1.1879e-04    &4.4599e-04               &3.5824e-08     &3.5679e-04 \\           \hline
1/128   &5.0193e-05   &2.9698e-05     &1.1151e-04             &2.2391e-09     &8.9208e-05\\           \hline
1/256    &1.2549e-05   &7.4246e-06   &2.7878e-05                &1.3994e-010    &2.2303e-05 \\ \hline
\hline
Rate  &2.00 & 2.00  &2.00  & 4.00  &2.00  \\
\hline
\end{tabular}
\end{center}
%\end{small}
}
\end{table}

\begin{table}[htbp]
\caption{Test Case 1:
Numerical performance of the WG scheme \eqref{WG-scheme} (domain $\Omega=(0,1)^2$, exact solution $u=\sin(\pi x)\sin(\pi y)$, uniform rectangular partitions, stabilization parameter $\rho=1$, mesh parameter $h=\max\{h_x,h_y\}$, and vanishing Dirichlet boundary data{\color{black}{).}}}\label{Example1:rho1:02}
{
\setlength{\extrarowheight}{1.5pt}
\begin{center}\centering\scriptsize
\begin{tabular}{|l|l|l|l|l|l|}
\hline $h$ &$\|u-\S(u_{b})\|_{\infty,\star}$& $\|u-\S(u_{b})\|_0$ & $\|\nabla_d({e_{b}})\|_0$  &$\|\nabla_d u_b-\nabla u\|_{0,\star}$  &    $\|\nabla(Q_{0}u-\S(u_{b}))\|_0$  \\
\hline
1/4&      2.6067e-02    & 2.1660e-02     &9.9161e-02                &5.6547e-02     &9.0642e-02  \\ \hline
1/8&    6.6802e-03    &5.4600e-03    &2.7845e-02               &1.8770e-02      &2.6179e-02\\ \hline
1/16&     1.6762e-03    &1.3665e-03     & 7.2033e-03                &5.0348e-03    &6.8182e-03  \\ \hline
1/32&     4.1934e-04    &3.4168e-04     & 1.8169e-03                &1.2811e-03     &1.7226e-03  \\ \hline
1/64&    1.0486e-04    &8.5423e-05    &4.5525e-04               &3.2168e-04      &4.3180e-04\\ \hline
1/128&   2.6215e-05   &2.1356e-05     &1.1388e-04                &8.0509e-05     &1.0802e-04\\ \hline
1/256    &6.5538e-06   &5.3390e-06   &2.8473e-05         &2.0133e-05    &2.7010e-05\\           \hline
\hline
Rate &2.00 &2.00     &2.00  & 2.00 & 2.00 \\
\hline
\end{tabular}
%\label{tab1}
%\end{small}
\end{center}
}
\end{table}

{\bf Test Case 2 (Nonhomogeneous BVP):} The model problem (\ref{a1})-(\ref{aa1}) is again defined on the unit square domain $\Omega=(0,1)^2$, and the diffusive coefficient $a$ is the identity matrix. The exact solution in this test case is given by $u =\sin(x)\cos(y)$. The Dirichlet boundary value is given by the restriction of the exact solution on the boundary, and the right-hand side function $f$ is computed accordingly.

The Dirichlet boundary value is clearly non-trivial so that the superconvergence estimate (\ref{EQ:SuperConv}) holds true when the numerical boundary value is chosen as the modified $L^2$ projection of the boundary data shown as in \eqref{EQ:PL2:01}-\eqref{EQ:PL2:02}. Tables \ref{Example2:rho1:SquareWP} and \ref{Example2:rho1:RectWP} illustrate the performance of the WG finite element scheme \eqref{WG-scheme} when the modified $L^2$ projection of the boundary data is employed in \eqref{WG-scheme}. The result shows a superconvergence of rate $r=2$ in the discrete $H^1$ norm, which is in great consistency with the theory.

\begin{table}[htbp]\centering\scriptsize
\caption{Test Case 2: Convergence of the lowest order WG-FEM on the unit square domain with exact solution $u=\sin(x)\cos(y)$, uniform square partitions, stabilization parameter $\rho=1$, and perturbed $L^2$ projection of the Dirichlet boundary data by \eqref{EQ:PL2:01}-\eqref{EQ:PL2:02}.}\label{Example2:rho1:SquareWP}
{
\setlength{\extrarowheight}{1.5pt}
\begin{center}
\begin{tabular}{|l|l|l|l|l|l|}
\hline $h$ &$\|u-\S(u_{b})\|_{\infty,\star}$& $\|u-\S(u_{b})\|_0$ & $\|\nabla_d({e_{b}})\|_0$  &$\|\nabla_d u_b-\nabla u\|_{0,\star}$  &    $\|\nabla(Q_{0}u-\S(u_{b}))\|_0$  \\
\hline
1/4&     1.1296e-02    &9.0829e-03     &1.9733e-02                &1.5754e-02     &1.8936e-02 \\ \hline
1/8&     3.0879e-03    &2.2783e-03    &4.9488e-03               &3.9540e-03      &4.7493e-03  \\ \hline
1/16&    8.0006e-04    &5.7037e-04     &1.2386e-03                &9.8992e-04     &1.1887e-03\\ \hline
1/32&    2.0299e-04    &1.4265e-04     &3.0975e-04                &2.4758e-04     &2.9728e-04  \\ \hline
1/64&    5.1074e-05    &3.5667e-05    &7.7444e-05               &6.1902e-05     &7.4327e-05 \\ \hline
1/128&   1.2806e-05   &8.9170e-06     &1.9362e-05                &1.5476e-05    &1.8582e-05 \\ \hline
1/256    &3.2058e-06   &2.2293e-06   &4.8404e-06                &3.8690e-06    &4.6456e-06 \\ \hline
\hline
Rate &2.00 & 2.00 &2.00  & 2.00  & 2.00  \\
\hline
\hline
\end{tabular}
\end{center}
}
%\end{small}
\end{table}

\begin{table}[htbp]\centering\scriptsize
\caption{Test Case 2: Convergence of the lowest order WG-FEM on the unit square domain with exact solution $u=\sin(x)\cos(y)$, uniform rectangular partitions, stabilization parameter $\rho=1$, mesh parameter $h=\max(h_x,h_y)$, and perturbed $L^2$ projection of the Dirichlet boundary data by \eqref{EQ:PL2:01}-\eqref{EQ:PL2:02}.}\label{Example2:rho1:RectWP}
{
\setlength{\extrarowheight}{1.5pt}
\begin{center}
\begin{tabular}{|l|l|l|l|l|l|}
\hline $h$ &$\|u-\S(u_{b})\|_{\infty,\star}$& $\|u-\S(u_{b})\|_0$ & $\|\nabla_d({e_{b}})\|_0$  &$\|\nabla_d u_b-\nabla u\|_{0,\star}$  &    $\|\nabla(Q_{0}u-\S(u_{b}))\|_0$  \\
\hline
1/4&      1.1667e-02    & 8.3456e-03     &1.4734e-02                & 1.1884e-02     &1.4008e-02  \\ \hline
1/8&     3.1447e-03    &2.1001e-03    &3.6994e-03               & 2.9877e-03      &3.5184e-03\\ \hline
1/16&     8.1077e-04    &5.2619e-04     &9.2630e-04               &7.4845e-04    &8.8109e-04  \\ \hline
1/32&     2.0539e-04    &1.3163e-04     & 2.3168e-04                &1.8722e-04     &2.2038e-04  \\ \hline
1/64&    5.1652e-05    &3.2912e-05    &5.7926e-05               &4.6813e-05      &5.5101e-05\\ \hline
1/128&   1.2949e-05   &8.2283e-06     &1.4482e-05                &1.1704e-05     &1.3776e-05\\ \hline
1/256    &3.2415e-06   &2.0571e-06   &3.6205e-06         &2.9259e-06    &3.4440e-06\\           \hline
\hline
Rate       &2.00  &2.00 &2.00 & 2.00 & 2.00 \\
\hline
\hline
\end{tabular}
\end{center}
}
%\end{small}
\end{table}

Tables \ref{Example2:rho1:SquareNP} and \ref{Example2:rho1:RectNP} show the performance of the WG finite element scheme \eqref{WG-scheme} when the exact $L^2$ projection of the boundary data is employed in \eqref{WG-scheme}. On uniform {\em square partitions}, the numerical solutions are seen to be convergent at the rate of $r=2$ in the discrete $H^1$ norm as shown in Table \ref{Example2:rho1:SquareNP}.
It should be pointed out that on uniform square partitions, one may carry out the analysis further to derive a superconvergence with the full rate of $r=2$ if the boundary data satisfies $u_{xx}=u_{yy}$, which is the case for Test Case 2. On the other hand, Table \ref{Example2:rho1:RectNP} illustrates a convergence at a rate lower than $r=2$ on uniform {\em rectangular partitions}. The result with rectangular partitions reveals a sub-optimal order of superconvergence for the scheme \eqref{WG-scheme} when the boundary value is approximated by the $L^2$ projection or the usual nodal point interpolation. This sub-optimal order has been theoretically proved to be $r=1.5$ with proper regularity assumptions on the exact solution. Note that the computation indicates a superconvergence with an order around $r=1.9$ rather than $r=1.5$.

\begin{table}[htbp]\centering\scriptsize
\caption{Test Case 2: Convergence of the lowest order WG-FEM on the unit square domain with exact solution $u=\sin(x)\cos(y)$, uniform square partitions, stabilization parameter $\rho=1$, mesh parameter $h=\max(h_x,h_y)$, and $L^2$ projection of the Dirichlet boundary data.}\label{Example2:rho1:SquareNP}
{
\setlength{\extrarowheight}{1.5pt}
\begin{center}
\begin{tabular}{|l|l|l|l|l|l|}
\hline $h$ &$\|u-\S(u_{b})\|_{\infty,\star}$& $\|u-\S(u_{b})\|_0$ & $\|\nabla_d({e_{b}})\|_0$  &$\|\nabla_d u_b-\nabla u\|_{0,\star}$  &    $\|\nabla(Q_{0}u-\S(u_{b}))\|_0$  \\
\hline
1/4&     7.8025e-03    &3.4806e-03     &8.0885e-04                &3.9540e-03     &1.0155e-03 \\ \hline
1/8&     2.0827e-03    &8.7414e-04    &2.0366e-04               &9.8946e-04      &2.5557e-04  \\ \hline
1/16&   5.3530e-04    &2.1884e-04     &5.1091e-05                &2.4754e-04     &6.4140e-05\\ \hline
1/32&    1.3547e-04    &5.4730e-05     &1.2788e-05                &6.1899e-05     &1.6056e-05  \\ \hline
1/64&    3.4060e-05    &1.3684e-05    &3.1980e-06               &1.5476e-05     &4.0154e-06 \\ \hline
1/128&   8.5379e-06   &3.4210e-06     &7.9956e-07                &3.8690e-06    &1.0039e-06 \\ \hline
1/256    &2.1372e-06   &8.5526e-07   &1.9989e-07                &9.6726e-07    &2.5111e-07 \\ \hline
\hline
Rate &2.00 & 2.00 &2.00 & 2.00 & 2.00 \\
\hline
\end{tabular}
\end{center}
}
%\end{small}
\end{table}

\begin{table}[htbp]\centering\scriptsize
\caption{Test Case 2: Convergence of the lowest order WG-FEM on the unit square domain with exact solution $u=\sin(x)\cos(y)$, uniform rectangular partitions, stabilization parameter $\rho=1$, mesh parameter $h=\max(h_x,h_y)$, and $L^2$ projection of the Dirichlet boundary data.}\label{Example2:rho1:RectNP}
{
\setlength{\extrarowheight}{1.5pt}
\begin{center}
\begin{tabular}{|l|l|l|l|l|l|}
\hline $h$ &$\|u-\S(u_{b})\|_{\infty,\star}$& $\|u-\S(u_{b})\|_0$ & $\|\nabla_d({e_{b}})\|_0$  &$\|\nabla_d u_b-\nabla u\|_{0,\star}$  &    $\|\nabla(Q_{0}u-\S(u_{b}))\|_0$  \\
\hline
1/4&      5.2855e-03    & 2.2786e-03     &6.3200e-03                & 6.8392e-03     &6.3871e-03  \\ \hline
1/8&     1.5497e-03    &5.7273e-04    &2.0952e-03               & 2.1923e-03      &2.1099e-03\\ \hline
1/16&     4.3514e-04    &1.4374e-04     &6.1241e-04               &6.3312e-04    &6.1571e-04  \\ \hline
1/32&     1.1742e-04    &3.5992e-05     &1.7103e-04                &1.7568e-04     &1.7179e-04  \\ \hline
1/64&    3.0790e-05    &9.0027e-06    &4.6707e-05               &4.7771e-05      &4.6880e-05\\ \hline
1/128&   7.9255e-06   &2.2511e-06     &1.2579e-05                &1.2826e-05     &1.2619e-05\\ \hline
1/256    &2.0190e-06   &5.6279e-07   &3.3544e-06         &3.4124e-06    &3.3640e-06\\           \hline
\hline
Rate  &1.97 &2.00     &1.91  & 1.91 & 1.91 \\
\hline
\hline
\end{tabular}
\end{center}
}
%\end{small}
\end{table}

Tables \ref{Example2:rho6:SquareWP}-\ref{Example2:rho6:01:RectNP} illustrate the numerical performance of the lowest order WG-FEM on uniform square or rectangular partitions with the stabilization parameter $\rho=6$. The results are similar to the case of $\rho=1$.

The following conclusions seem to be appropriate from Test Case 2:
\begin{itemize}
\item If the Dirichlet boundary value is approximated by the modified $L^2$ projection, the theory-predicted superconvergence of order $2$ by Corollary \ref{Corollary:SuperC-nonhomo} is computationally valid for the WG scheme (\ref{WG-scheme}).
    \item If the Dirichlet boundary value is approximated by the exact $L^2$ projection, the numerical solutions of (\ref{WG-scheme}) do not have a full rate of convergence at $r=2$, but a convergence at a lower rate of $r \approx 1.9$ is observed numerically. Hence, the computation outperforms the theoretical superconvergence of $r=1.5$.
\end{itemize}

\begin{table}[htbp]\centering\scriptsize
\caption{Test Case 2: Convergence of the lowest order WG-FEM on the unit square domain with exact solution $u=\sin(x)\cos(y)$, uniform square partitions, stabilization parameter $\rho=6$, mesh parameter $h=(h_x+h_y)/2$, perturbed $L^2$ projection of the Dirichlet boundary data by \eqref{EQ:PL2:01}-\eqref{EQ:PL2:02}.}\label{Example2:rho6:SquareWP}
{
\setlength{\extrarowheight}{1.5pt}
\begin{center}
\begin{tabular}{|l|l|l|l|l|l|}
\hline $h$ &$\|u-\S(u_{b})\|_{\infty,\star}$& $\|u-\S(u_{b})\|_0$ & $\|\nabla_d({e_{b}})\|_0$  &$\|\nabla_d u_b-\nabla u\|_{0,\star}$  &    $\|\nabla(Q_{0}u-\S(u_{b}))\|_0$  \\
\hline

\hline
1/4     &7.8248e-03    &3.4864e-03     &7.9599e-04                & 3.9619e-03  &1.0135e-03   \\    \hline
1/8     &2.0842e-03    &8.7474e-04    &2.0322e-04               &9.9044e-04    &2.5604e-04 \\     \hline
1/16   &5.3544e-04    &2.1888e-04     &5.1080e-05                &2.4762e-04    &6.4192e-05 \\      \hline
1/32    &1.3548e-04    &5.4733e-05     &1.2788e-05                &6.1904e-05    &1.6060e-05  \\       \hline
1/64    &3.4060e-05    &1.3684e-05    &3.1980e-06               &1.5476e-05     &4.0156e-06 \\           \hline
1/128   &8.5379e-06   &3.4210e-06     &7.9956e-07               &3.8690e-06     &1.0040e-06\\           \hline
1/256    &2.1373e-06   &8.5526e-07   &1.9989e-07                &9.6726e-07    &2.5111e-07 \\ \hline
\hline
Rate  &2.00 & 2.00 &2.00  & 2.00 &2.00 \\
\hline
\end{tabular}
\end{center}
}
%\end{small}
\end{table}

\begin{table}[htbp]\centering\scriptsize
\caption{Test Case 2: Convergence of the lowest order WG-FEM on the unit square domain with exact solution $u=\sin(x)\cos(y)$, uniform rectangular partitions, stabilization parameter $\rho=6$, $h=(h_x+h_y)/2$, and perturbed $L^2$ projection of the Dirichlet boundary data by \eqref{EQ:PL2:01}-\eqref{EQ:PL2:02}.}\label{Example2:rho6:RectWP}
{
\setlength{\extrarowheight}{1.5pt}
\begin{center}
\begin{tabular}{|l|l|l|l|l|l|}
\hline $h$ &$\|u-\S(u_{b})\|_{\infty,\star}$& $\|u-\S(u_{b})\|_0$ & $\|\nabla_d({e_{b}})\|_0$  &$\|\nabla_d u_b-\nabla u\|_{0,\star}$  &    $\|\nabla(Q_{0}u-\S(u_{b}))\|_0$  \\
\hline
2.08e-01     &5.6939e-03    &2.4689e-03     &7.5368e-04                &2.5643e-03   &6.6175e-04   \\    \hline
1.04e-01      &1.5092e-03    &6.1891e-04     &1.9018e-04                &6.4106e-04   &1.6697e-04 \\     \hline
5.21e-02    &3.8698e-04    &1.5483e-04     &4.7657e-05                &1.6027e-04    &4.1846e-05 \\      \hline
2.60e-02    &9.7868e-05    &3.8715e-05     &1.1921e-05                &4.0067e-05    &1.0468e-05  \\       \hline
1.30e-02    &2.4601e-05    &9.6791e-06    &2.9808e-06               &1.0017e-05     &2.6174e-06      \\           \hline
6.51e-03   &6.1664e-06   &2.4198e-06   &7.4522e-07                 &2.5042e-06  &6.5438e-07  \\           \hline
3.26e-03    &1.5436e-06   &6.0495e-07   &1.8630e-07         &6.2605e-07    &1.6451e-07\\           \hline
\hline
Rate &2.00 & 2.00 &2.00 & 2.00 & 2.00   \\
\hline
\end{tabular}
\end{center}
}
%\end{small}
\end{table}

\begin{table}[htbp]\centering\scriptsize
\caption{Test Case 2: Convergence of the lowest order WG-FEM on the unit square domain with exact solution $u=\sin(x)\cos(y)$, uniform rectangular partitions, stabilization parameter $\rho=6$, mesh parameter $h=(h_x+h_y)/2$, and $L^2$ projection of the Dirichlet boundary data.}\label{Example2:rho6:01:RectNP}
{
\setlength{\extrarowheight}{1.5pt}
\begin{center}
\begin{tabular}{|l|l|l|l|l|l|}
\hline $h$ &$\|u-\S(u_{b})\|_{\infty,\star}$& $\|u-\S(u_{b})\|_0$ & $\|\nabla_d({e_{b}})\|_0$  &$\|\nabla_d u_b-\nabla u\|_{0,\star}$  &    $\|\nabla(Q_{0}u-\S(u_{b}))\|_0$  \\
\hline
2.08e-01     &5.6665e-03    &2.4609e-03     &1.8508e-03                &3.3965e-03   &1.9430e-03   \\    \hline
1.04e-01      &1.5383e-03    &6.1680e-04     &5.1792e-04                &8.7958e-04   &5.3789e-04 \\     \hline
5.21e-02    &4.0252e-04    &1.5430e-04     &1.4171e-04               &2.2726e-04    &1.4624e-04 \\      \hline
2.60e-02    &1.0309e-04    &3.8582e-05     &3.8231e-05                &5.8602e-05    &3.9279e-05  \\       \hline
1.30e-02    &2.6126e-05    &9.6460e-06    &1.0210e-05               &1.5084e-05     &1.0456e-05      \\           \hline
6.51e-03   &6.5830e-06   &2.4115e-06   &2.7058e-06                 &3.8763e-06  &2.7637e-06  \\           \hline
3.26e-03    &1.6532e-06   &6.0288e-07   &7.1270e-07         &9.9473e-07    &7.2655e-07\\           \hline
\hline
Rate  &2.00 & 2.00 &1.92 & 1.96 & 1.93 \\
\hline
\end{tabular}
\end{center}
}
%\end{small}
\end{table}

{\bf Test Case 3 (Nonhomogeneous BVP):} The model problem in Test Case 3 has exact solution $u=\exp(x)\sin(y)$ on the unit square domain, with diffusive coefficient tensor given as the identity matrix. Unlike Test Case 2, the boundary data for Test Case 3 does not satisfy $u_{xx}=u_{yy}$ so that no superconvergence of full order of $r=2$ is predicted even on uniform square partitions when the usual $L^2$ projection is applied for the boundary data.

Table \ref{Example3:rho1:SquareNP} shows the numerical result on uniform square partitions when the exact $L^2$ projection of the Dirichlet boundary value is used in the numerical scheme \eqref{WG-scheme}. The result indicates a superconvergence of sub-optimal order of $r\approx 1.9$. Again, it should be pointed out that a superconvergence of order $r=1.5$ has been theoretically established in previous sections.

\begin{table}[htbp]\centering\scriptsize
\caption{Test Case 3: Convergence of the lowest order WG-FEM on the unit square domain with exact solution $u=\exp(x)\sin(y)$, uniform square partitions, stabilization parameter $\rho=1$, $L^2$ projection of the boundary data.}\label{Example3:rho1:SquareNP}
{
\setlength{\extrarowheight}{1.5pt}
\begin{center}
\begin{tabular}{|l|l|l|l|l|l|}
\hline $h$ &$\|u-\S(u_{b})\|_{\infty,\star}$& $\|u-\S(u_{b})\|_0$ & $\|\nabla_d({e_{b}})\|_0$  &$\|\nabla_d u_b-\nabla u\|_{0,\star}$  &    $\|\nabla(Q_{0}u-\S(u_{b}))\|_0$  \\
\hline
$1/4$     &1.6582e-02    &1.3205e-02     &8.726631e-02                &8.726721e-02    &8.7277e-02  \\     \hline
$1/8$     &7.3819e-03    &3.7783e-03     &3.051025e-02                &3.051030e-02    &3.0512e-02 \\           \hline
$1/16$    &2.3888e-03    &9.9876e-04     &9.281043e-03                &9.281047e-03    &9.2813e-03  \\     \hline
$1/32$    &7.0761e-04    &2.5462e-04     &2.668696e-03                &2.668696e-03    &2.6688e-03  \\           \hline
$1/64$    &1.9432e-04    &6.4063e-05     &7.440802e-04                &7.440803e-04     &7.4409e-04 \\              \hline
$1/128$   &5.1486e-05    &1.6047e-05     &2.034332e-04                &2.034332e-04     &2.0344e-04\\           \hline
$1/256$    &1.3370e-05   &4.0143e-06     &5.486702e-05                &5.486703e-05      & 5.4869e-05                         \\           \hline
\hline
Rate  &1.95  &2.00  &1.89 &1.89 & 1.89 \\
\hline
\end{tabular}
\end{center}
}
%\label{tab1-06}
%\end{small}
\end{table}

Table \ref{Example3:rho1:SquareWP} shows the numerical result on uniform square partitions when the modified $L^2$ projection of the Dirichlet boundary value is used in the numerical scheme \eqref{WG-scheme}. The result indicates a full superconvergence of rate $r=2$, which is in great consistency with Corollary \ref{Corollary:SuperC-nonhomo}.

\begin{table}[htbp]\centering\scriptsize
\caption{Test Case 3: Convergence of the lowest order WG-FEM on the unit square domain with exact solution $u=\exp(x)\sin(y)$, uniform square partitions, $\rho=1$, perturbed $L^2$ projection of the boundary data by \eqref{EQ:PL2:01}-\eqref{EQ:PL2:02}.}\label{Example3:rho1:SquareWP}
{
\setlength{\extrarowheight}{1.5pt}
\begin{center}
\begin{tabular}{|l|l|l|l|l|l|}
\hline $h$ &$\|u-\S(u_{b})\|_{\infty,\star}$& $\|u-\S(u_{b})\|_0$ & $\|\nabla_d({e_{b}})\|_0$  &$\|\nabla_d u_b-\nabla u\|_{0,\star}$  &    $\|\nabla(Q_{0}u-\S(u_{b}))\|_0$  \\
\hline
$1/4$     &2.5301e-03    &9.2574e-03     &4.554895e-02                &4.551208e-02    &4.7385e-02  \\     \hline
$1/8$     &6.3910e-04    &2.1974e-03     &1.143912e-02                &1.143926e-02      &1.1900e-02 \\           \hline
$1/16$    &1.6260e-04    &5.4172e-04     &2.863968e-03                &2.863977e-03    &2.9792e-03  \\     \hline
$1/32$    &4.0828e-05    &1.3495e-04     &7.162828e-04                &7.162834e-04    &7.4509e-04  \\           \hline
$1/64$    &1.0211e-05    &3.3706e-05     &1.790895e-04                &1.790896e-04      &1.8629e-04 \\              \hline
$1/128$   &2.5534e-06   &8.4246e-06      &4.477358e-05                &4.477358e-05     &4.6574e-05\\           \hline
$1/256$    &6.3837e-07   &2.1060e-06     &1.119347e-05                &1.119347e-05      & 1.1619e-05                         \\           \hline
\hline
Rate       &2.00 & 2.00 &2.00 &2.00 & 2.00 \\
\hline
\hline
\end{tabular}
\end{center}
}
%\label{tab1-06}
%\end{small}
\end{table}

Tables \ref{Example3:rho1:RectNP}-\ref{Example3:rho1:RectWP} show the numerical results on uniform rectangular partitions when the Dirichlet boundary value is approximated with various approaches in the numerical scheme \eqref{WG-scheme}. The results are similar to those on square partitions. Most notably, Table \ref{Example3:rho1:RectWP} provides a numerical verification of the superconvergence estimate detailed in Corollary \ref{Corollary:SuperC-nonhomo}.

\begin{table}[htbp]\centering\scriptsize
\caption{Test Case 3: Convergence of the lowest order WG-FEM on the unit square domain with exact solution $u=\exp(x)\sin(y)$, uniform rectangular partitions, stabilization parameter $\rho=1$, meshsize $h=\max\{h_x,h_y\}$, and $L^2$ projection of the boundary data.}\label{Example3:rho1:RectNP}
{
\setlength{\extrarowheight}{1.5pt}
\begin{center}
\begin{tabular}{|l|l|l|l|l|l|}
\hline $h$ &$\|u-\S(u_{b})\|_{\infty,\star}$& $\|u-\S(u_{b})\|_0$ & $\|\nabla_d({e_{b}})\|_0$  &$\|\nabla_d u_b-\nabla u\|_{0,\star}$  &    $\|\nabla(Q_{0}u-\S(u_{b}))\|_0$  \\
\hline
1/4    &1.6184e-02    &1.0739e-02    &7.9121e-02                &7.8678e-02   &7.9028e-02   \\    \hline
1/8     &6.8896e-03    &3.1242e-03    &2.6904e-02               &2.6796e-02    &2.6881e-02 \\     \hline
1/16    &2.3176e-03    &8.2994e-04     &8.1057e-03                &8.0813e-03    &8.1005e-03 \\      \hline
1/32    &6.7670e-04    &2.1190e-04     &2.3194e-03                &2.3139e-03    &2.3182e-03  \\       \hline
1/64    &1.8508e-04    &5.3335e-05    &6.4474e-04               &6.4349e-04     &6.4447e-04      \\           \hline
1/128   &4.8913e-05   &1.3362e-05   &1.7591e-04           &1.7563e-04  &1.7585e-04  \\           \hline
1/256    &1.2678e-05   &3.3426e-06   &4.7374e-05         &4.7308e-05    &4.7361e-05\\           \hline
\hline
Rate     &1.95  &2.00     &1.89  & 1.89  & 1.89    \\
\hline
\end{tabular}
\end{center}
}
%\label{tab1-06}
%\end{small}
\end{table}

\begin{table}[htbp]\centering\scriptsize
\caption{Test Case 3: Convergence of the lowest order WG-FEM on the unit square domain with exact solution $u=\exp(x)\sin(y)$, uniform rectangular partitions, stabilization parameter $\rho=1$, meshsize $h=\max\{h_x,h_y\}$, and perturbed $L^2$ projection of the boundary data by \eqref{EQ:PL2:01}-\eqref{EQ:PL2:02}.}\label{Example3:rho1:RectWP}
{
\setlength{\extrarowheight}{1.5pt}
\begin{center}
\begin{tabular}{|l|l|l|l|l|l|}
\hline $h$ &$\|u-\S(u_{b})\|_{\infty,\star}$& $\|u-\S(u_{b})\|_0$ & $\|\nabla_d({e_{b}})\|_0$  &$\|\nabla_d u_b-\nabla u\|_{0,\star}$  &    $\|\nabla(Q_{0}u-\S(u_{b}))\|_0$  \\
\hline
1/4    &1.0775e-02    &8.5720e-03     &4.2529e-02                &4.1088e-02   &4.3546e-02   \\    \hline
1/8     &3.0986e-03    &2.1106e-03    &1.0669e-02               &1.0307e-02    &1.0923e-02 \\     \hline
1/16    &8.2422e-04    &5.2608e-04     &2.6702e-03                &2.5797e-03    &2.7340e-03 \\      \hline
1/32    &2.1187e-04    &1.3144e-04     &6.6775e-04                &6.4512e-04    &6.8370e-04  \\       \hline
1/64    &5.3652e-05    &3.2854e-05    &1.6695e-04               &1.6129e-04     &1.7094e-04      \\           \hline
1/128   &1.3495e-05   &8.2132e-06   &4.1739e-05           &4.0324e-05  &4.2736e-05  \\           \hline
1/256    &3.3836e-06   &2.0532e-06   &1.0435e-05         &1.0081e-05    &1.0683e-05\\           \hline
\hline
Rate      &2.00  &2.00 &2.00 & 2.00 & 2.00 \\
\hline
\end{tabular}
\end{center}
}
%\label{tab1-06}
%\end{small}
\end{table}

{\bf Test Case 4 (Nonhomogeneous BVP):} Table \ref{Example4:rho1:RectWP} contains some numerical results for the model problem in $\Omega=(0,1)^{2}$ with exact solution $u=\sin(x)\sin(y)$ on uniform rectangular partitions. The diffusive coefficient tensor $a$ is a full symmetric and positive definite matrix with constant values. The results are in consistency with our theory.

\begin{table}[htbp]\centering\scriptsize
{\color{black}{\caption{Test Case 4: Convergence of the lowest order WG-FEM on the unit square domain with exact solution $u=\sin(x)\sin(y)$, uniform rectangular partitions, stabilization parameter $\rho=1$, meshsize $h=\max\{h_x,h_y\}$, and $L^2$ projection of the boundary data, coefficient matrix  $a_{11}=3$, $a_{12}=a_{21}=1$, and $a_{22}=2$.} \label{Example4:rho1:RectWP}
{
\setlength{\extrarowheight}{1.5pt}
\begin{center}
\begin{tabular}{|l|l|l|l|l|l|}
\hline $h$ &$\|u-\S(u_{b})\|_{\infty,\star}$& $\|u-\S(u_{b})\|_0$ & $\|\nabla_d({e_{b}})\|_0$  &$\|\nabla_d u_b-\nabla u\|_{0,\star}$  &    $\|\nabla(Q_{0}u-\S(u_{b}))\|_0$  \\
\hline
1/4     &5.6374e-03    &3.6092e-03     &1.6962e-02                &1.7087e-02    &1.6982e-02   \\    \hline
1/8     &2.2290e-03    &1.0981e-03     &6.8691e-03                &6.8863e-03    &6.8733e-03 \\     \hline
1/16    &7.4201e-04    &3.0150e-04     &2.2557e-03                &2.2588e-03    &2.2566e-03 \\      \hline
1/32    &2.1833e-04    &7.8118e-04     &6.7364e-04                &6.7429e-04    &6.7385e-04  \\       \hline
1/64    &6.0001e-05    &1.9769e-05     &1.9175e-04                &1.9190e-04    &1.9180e-04      \\           \hline
1/128   &1.5811e-05    &4.9615e-06     &5.3108e-05                &5.3141e-05    &5.3119e-05  \\           \hline
1/256   &4.0795e-06    &1.2419e-06     &1.4452e-05                &1.4459e-05    &1.4454e-05\\           \hline
\hline
Rate      &1.95        &2.00         &1.88                         & 1.88     & 1.88 \\
\hline
\end{tabular}
\end{center}
}}
}

%\label{tab1-06}
%\end{small}
\end{table}

{\bf Test Case 5 (Nonhomogeneous BVP):} In this example, the domain is the unit square $\Omega=(0,1)^{2}$ and the diffusive coefficient matrix is the identity. The exact solution is chosen as $u=\cos(\pi x)\sin(\pi y)$. The rectangular partition %(see Fig.\ref{fig1 non-uniform}) 
was obtained as the tensor product of two one-dimensional partitions in $x$ and $y$ directions, respectively. The 1-d non-uniform mesh in the $x$ direction is given by $[0:h_1:0.5,0.5:h_2:1]$ with $h_2=h_1/2$, and the 1-d non-uniform mesh in the $y$ direction is given by $[0:\tau_1:0.5,0.5:\tau_2:1]$ with $\tau_2=\tau_1/2$; see \cite{SYM-LI} for more details. Tables \ref{Example7:rho1:non-RectNP1-1}-\ref{Example7:rho1:non-RectNP1-2} illustrate the superconvergence performance for the WG finite element approximations. The results are in good consistency with the theory.

%\begin{figure}[!ht]
%\centering
%\includegraphics[height=0.60\textwidth,width=0.9\textwidth]{fig2.png}
%\caption{A non-uniform rectangular partition with $h_1=\tau_1=1/32$.}
%\label{fig1 non-uniform}
%\end{figure}

\begin{table}[htbp]\centering\scriptsize
{\color{black}{\caption{Test Case 5: Convergence of the lowest order WG-FEM on the unit square domain with exact solution $u=\cos(\pi x)\sin(\pi y)$, nonuniform rectangular partitions, stabilization parameter $\rho=1$, and $L^2$ projection of the boundary data.} \label{Example7:rho1:non-RectNP1-1}
{
\setlength{\extrarowheight}{1.5pt}
\begin{center}
\begin{tabular}{|l|l|l|l|l|l|}
\hline $h$ &$\|u-\S(u_{b})\|_{\infty,\star}$& $\|u-\S(u_{b})\|_0$ & $\|\nabla_d({e_{b}})\|_0$  &$\|\nabla_d u_b-\nabla u\|_{0,\star}$  &    $\|\nabla(Q_{0}u-\S(u_{b}))\|_0$  \\
\hline
1/4     &6.5626e-02    &2.3170e-02     &9.2942e-02                &8.9523e-02    &8.8819e-02   \\    \hline
1/8     &1.7693e-02    &6.0733e-03     &3.5040e-02                &3.4625e-02    &3.4636e-02 \\     \hline
1/16    &5.1259e-03    &1.5556e-03     &1.0803e-02                &1.0753e-02    &1.0752e-02 \\      \hline
1/32    &1.4092e-03    &3.9244e-04     &3.0306e-03                &3.0224e-03    &3.0217e-03  \\       \hline
1/64    &3.7339e-04    &9.8396e-05     &8.2070e-04                &8.1906e-04    &8.1883e-04      \\           \hline
1/128   &9.6152e-05    &2.4621e-05     &2.1888e-04                &2.1851e-04    &2.1845e-04  \\           \hline
1/256   &2.4448e-05    &6.1567e-06     &5.7874e-05                &5.7789e-05    &5.7775e-05\\           \hline
\hline
Rate     &1.98       &2.00           &1.92     & 1.92  & 1.92    \\
\hline
\end{tabular}
\end{center}
}}
}

%\label{tab1-06}
%\end{small}
\end{table}

\begin{table}[htbp]\centering\scriptsize
{\color{black}{\caption{Test Case 5: Convergence of the lowest order WG-FEM on the unit square domain with exact solution $u=\cos(\pi x)\sin(\pi y)$, nonuniform rectangular partitions, stabilization parameter $\rho=1$, and perturbed $L^2$ projection of the boundary data.} \label{Example7:rho1:non-RectNP1-2}
{
\setlength{\extrarowheight}{1.5pt}
\begin{center}
\begin{tabular}{|l|l|l|l|l|l|}
\hline $h$ &$\|u-\S(u_{b})\|_{\infty,\star}$& $\|u-\S(u_{b})\|_0$ & $\|\nabla_d({e_{b}})\|_0$  &$\|\nabla_d u_b-\nabla u\|_{0,\star}$  &    $\|\nabla(Q_{0}u-\S(u_{b}))\|_0$  \\
\hline
1/4     &1.1206e-01    &6.2495e-02     &3.4427e-01                &2.7957e-01    &3.3185e-01   \\    \hline
1/8     &3.4880e-02    &1.6893e-02     &9.3762e-02                &7.7517e-02    &9.0669e-02 \\     \hline
1/16    &9.4654e-03    &4.3372e-03     &2.4119e-02                &2.0066e-02    &2.3350e-02 \\      \hline
1/32    &2.4449e-03    &1.0922e-03     &6.0769e-03                &5.0643e-03    &5.8849e-03  \\       \hline
1/64    &6.1982e-04    &2.7356e-04     &1.5222e-03                &1.2691e-03    &1.4743e-03      \\           \hline
1/128   &1.5595e-04    &6.8422e-05     &3.8075e-04                &3.1748e-04    &3.6876e-04  \\           \hline
1/256   &3.9104e-05    &1.7108e-05     &9.5200e-05                &7.9382e-05    &9.2202e-05\\           \hline
\hline
Rate     &2.00       &2.00           &2.00     & 2.00  & 2.00    \\
\hline
\end{tabular}
\end{center}
}}
}

%\label{tab1-06}
%\end{small}
\end{table}

{\bf Test Case 6 (Nonhomogeneous BVP):} The configuration for this test example is as follows: (1) the domain is the unit square, (2) the diffusive coefficient tensor is the identity matrix, and (3) the exact solution is given by $u=\cos(\pi x)\sin(\pi y)$. The nonuniform rectangular partitions are obtained by perturbing the uniform $N\times N$ square partition with a random noise. More precisely, for any element $T=[x_i, x_{i+1}]\times [y_j, y_{j+1}]$ of the uniform $N\times N$ square partition of the domain $\Omega=(0,1)^2$, one alters $x_{i+1}$ and $y_{j+1}$ by using the following formula:
$$
x_{i+1}^*=x_{i+1} + 0.2(\mbox{rand}(1)-0.5)h,\qquad y_{j+1}^*=y_{j+1} + 0.2(\mbox{rand}(1)-0.5)h,
$$
where $h=1/N$ and $\mbox{rand}(1)$ is the MatLab function that returns a single uniformly distributed random number in the interval $(0,1)$. The WG finite element method of the lowest order was then employed to solve the model problem on each perturbed partition. Tables \ref{Example2:rho1:SquareWP:test4-1} and \ref{Example2:rho1:SquareWP:test4-2} illustrate the performance of the WG finite element method on such nonuniform partitions. The numerical results are in great consistency with the theory developed in previous sections.

\begin{table}[htbp]\centering\scriptsize
\caption{Test Case 6: Convergence of the lowest order WG-FEM on the unit square domain with exact solution $u=\cos(\pi x)\sin(\pi y)$, nonuniform square partitions, stabilization parameter $\rho=1$, and $L^2$ projection of the Dirichlet boundary data.}\label{Example2:rho1:SquareWP:test4-1}
{
\setlength{\extrarowheight}{1.5pt}
\begin{center}
\begin{tabular}{|l|l|l|l|l|l|}
\hline meshes &$\|u-\S(u_{b})\|_{\infty,\star}$& $\|u-\S(u_{b})\|_0$ & $\|\nabla_d({e_{b}})\|_0$  &$\|\nabla_d u_b-\nabla u\|_{0,\star}$  &    $\|\nabla(Q_{0}u-\S(u_{b}))\|_0$  \\
\hline
$8\times8$     &2.2983e-02    &8.8478e-03    &1.8961e-02               &2.4106e-02     &1.6519e-02\\     \hline
$16\times16$  &6.1184e-03    &2.2584e-03     &5.3835e-03                &6.8523e-03    &4.9258e-03 \\      \hline
$32\times32$   &1.7775e-03    &5.6631e-04     &1.9757e-03                &2.2323e-03   &1.9003e-03   \\       \hline
$64\times64$   &4.3652e-04    &1.4188e-04    &5.6026e-04               &6.1316e-04     &5.4222e-04 \\           \hline
$128\times128$  &1.0967e-04   &3.5501e-05     &1.9018e-04             &2.0005e-04     &1.8698e-04\\           \hline
\hline
Rate        &2.06          &2.06      &1.61                      & 1.67  &1.59 \\
\hline
\end{tabular}
\end{center}
}
\end{table}
%\end{small}

\begin{table}[htbp]\centering\scriptsize
\caption{Test Case 6: Convergence of the lowest order WG-FEM on the unit square domain with exact solution $u=\cos(\pi x)\sin(\pi y)$, nonuniform square partitions, stabilization parameter $\rho=1$, and perturbed $L^2$ projection of the Dirichlet boundary data by \eqref{EQ:PL2:01}-\eqref{EQ:PL2:02}.}\label{Example2:rho1:SquareWP:test4-2}
{
\setlength{\extrarowheight}{1.5pt}
\begin{center}
\begin{tabular}{|l|l|l|l|l|l|}
\hline $meshes$ &$\|u-\S(u_{b})\|_{\infty,\star}$& $\|u-\S(u_{b})\|_0$ & $\|\nabla_d({e_{b}})\|_0$  &$\|\nabla_d u_b-\nabla u\|_{0,\star}$  &    $\|\nabla(Q_{0}u-\S(u_{b}))\|_0$  \\
\hline
$8\times8$     &3.4671e-02    &2.2238e-02  &1.3155e-01      &1.0632e-01                   &1.2631e-01\\     \hline
$16\times16$  &1.0484e-02    &5.8001e-03   &3.4684e-02      &2.8371e-02                   &3.3389e-02 \\      \hline
$32\times32$   &3.0887e-03    &1.5254e-03  &9.0749e-03       &7.4953e-03                  &8.7541e-03   \\       \hline
$64\times64$   &7.8391e-04    &3.7947e-04  &2.2585e-03      &1.8614e-03                  &2.1775e-03 \\           \hline
$128\times128$  &2.0804e-04   &9.7546e-05   &5.7944e-04      &4.8025e-04             &5.5927e-04\\           \hline
\hline
Rate        &1.98          & 2.02      &2.02                                 & 2.02  &2.02 \\
\hline
\end{tabular}
\end{center}
}
\end{table}
%\end{small}

{\bf Test Case 7 (Nonhomogeneous BVP):} The configuration for this test example is the same as the Test Case 6, but the computation on the rate of convergence follows a different approach. More precisely, for each $N=2^j$, we first construct a nonuniform rectangular partition of size $N\times N$ by using the perturbation method employed in Test Case 6, and then obtain a numerical solution by using the lowest order WG finite element method. Next, we refine this nonuniform rectangular partition through the usual bisection method (i.e.,  divide each rectangular element into four equal-sized sub-rectangles), and then subsequently apply the WG-FEM on the new mesh. The two numerical solutions are used to compute the rate of convergence.

Table \ref{Example5:rho1:rectangularWP:testmesh5-2-6-17:9:00} shows the performance of the WG finite element method with the current configuration. The theoretical rate of convergence in the discrete $H^1$ norm is $r=1.5$, and the numerical experiment provides a good confirmation of the theory.

\begin{table}[htbp]\centering\scriptsize
\caption{Test Case 7: Convergence of the lowest order WG-FEM on the unit square domain with exact solution $u=\cos(\pi x)\sin(\pi y)$, nonuniform rectangular partitions, stabilization parameter $\rho=1$, and $L^2$ projection of the Dirichlet boundary data.}\label{Example5:rho1:rectangularWP:testmesh5-2-6-17:9:00}
{
\setlength{\extrarowheight}{1.5pt}
\begin{center}
\begin{tabular}{|l|l|l|l|l|l|}
\hline $N=$ &$\|u-\S(u_{b})\|_{\infty,\star}$& $\|u-\S(u_{b})\|_0$ & $\|\nabla_d({e_{b}})\|_0$  &$\|\nabla_d u_b-\nabla u\|_{0,\star}$  &    $\|\nabla(Q_{0}u-\S(u_{b}))\|_0$  \\
\hline
$2^2$    &6.8448e-02    &3.3954e-02     &6.7756e-02               &8.2546e-02     &5.4303e-02\\     \hline
$2^3$    &2.0911e-02    &8.7566e-03     &1.7788e-02                &2.3043e-02    &1.5040e-02 \\      \hline
Rate              &1.71          & 1.96         &1.93                      & 1.84       &1.85 \\ \hline
\hline
$2^3$    &2.3296e-02    &9.0039e-03     &1.8438e-02               &2.5500e-02     &1.6516e-02\\     \hline
$2^4$    &6.6533e-03    &2.2813e-03     &5.3720e-03                &7.0895e-03    &5.0103e-03 \\      \hline
Rate              &1.81          & 1.98         &1.78                      & 1.85       &1.72 \\ \hline
\hline
$2^4$    &6.5612e-03    &2.2503e-03     &5.2747e-03               &6.6620e-03     &4.7837e-03\\     \hline
$2^5$    &1.7191e-03    &5.6465e-04     &1.4668e-03                &1.7944e-03    &1.3604e-03 \\      \hline
Rate        &1.93                & 2.00      &1.85                                  & 1.89  &1.81 \\
\hline
$2^5$    &1.6922e-03    &5.7132e-04     &2.0170e-03               &2.2462e-03     &1.9364e-03\\     \hline
$2^6$    &4.4502e-04    &1.4304e-04     &6.4701e-04                &6.9347e-04    &6.3223e-04 \\      \hline
Rate             &1.93                & 2.00      &1.64                                  & 1.70  &1.61 \\
\hline
$2^6$    &4.4183e-04    &1.4148e-04     &5.3399e-04               &5.8978e-04     &5.1567e-04\\     \hline
$2^7$    &1.1448e-04    &3.5384e-05     &1.7216e-04                &1.8321e-04    &1.6865e-04 \\      \hline
Rate        &1.95                & 2.00      &1.63                                  & 1.69  &1.61 \\
\hline
$2^7$    &1.0699e-04    &3.5341e-05     &1.7135e-04               &1.8289e-04     &1.6785e-04\\     \hline
$2^8$    &2.7700e-05    &8.8366e-06     &5.7743e-05                &5.9914e-05    &5.7097e-05 \\      \hline
Rate        &1.95                & 2.00      &1.57                                  & 1.61  &1.56 \\
\hline
\end{tabular}
\end{center}
}
\end{table}
%\end{small}

\subsection{Numerical experiments with discontinuous coefficients}\label{subSection:NE8} The goal here is to numerically verify the superconvergence theory when the diffusive coefficient tensor is discontinuous in the domain.

{\bf Test Case 8 (Homogeneous BVP):} In this test case, the domain is given by $\Omega=(-1,1)^{2}$ and the diffusive coefficient tensor $a$ is given by
\begin{eqnarray*}
a=\left(\begin{array}{cccc}
\alpha_{i}^{x},0 \\
0,\alpha_{i}^{y}
\end{array}\right).
\end{eqnarray*}
The exact solution is chosen as $u=\alpha_{i}\sin(2\pi x)\sin(2\pi y)$. Here the value of the coefficient
 $\alpha^{x}_i,\alpha^{y}_i,$ $\alpha_{i}$ are specified in Table \ref{Example88:coefficient parameter}. This test problem has been considered in \cite{WG-CFO-Wang}. The numerical results are shown in Table \ref{Example5:rho1:SquWP}.

\begin{table}[htbp]\centering\scriptsize
{\color{black}{\caption{Test Case 7: Parameter values for the diffusive coefficients and the exact solution.} \label{Example88:coefficient parameter}
{
\setlength{\extrarowheight}{1.5pt}
\begin{center}
\begin{tabular}{|l|l|}
\hline
$\alpha_{4}^{x}=0.1$   &$\alpha_{3}^{x}=1000$\\
$\alpha_{4}^{y}=0.01$    &$\alpha_{3}^{y}=100$\\
$\alpha_{4}=100$    &$\alpha_{3}=0.01$    \\ \hline
$\alpha_{1}^{x}=100$   &$\alpha_{2}^{x}=1$\\
$\alpha_{1}^{y}=10$   &$\alpha_{2}^{y}=0.1$\\
$\alpha_{1}=0.1$   &$\alpha_{2}=10$    \\  \hline
\end{tabular}
\end{center}
}}
}
\end{table}

\begin{table}[htbp]\centering\scriptsize
{\color{black}{\caption{Test Case 8: Convergence of the lowest order WG-FEM on the $(-1,1)^{2}$ with exact solution $u=\alpha_{i}\sin(2\pi x)\sin(2\pi y)$, discontinuous diffusive tensor, uniform square partitions, stabilization parameter $\rho=1$, and $L^2$ projection of the boundary data.} \label{Example5:rho1:SquWP}
{
\setlength{\extrarowheight}{1.5pt}
\begin{center}
\begin{tabular}{|l|l|l|l|l|l|}
\hline $h$ &$\|u-\S(u_{b})\|_{\infty,\star}$& $\|u-\S(u_{b})\|_0$ & $\|\nabla_d({e_{b}})\|_0$  &$\|\nabla_d u_b-\nabla u\|_{0,\star}$  &    $\|\nabla(Q_{0}u-\S(u_{b}))\|_0$  \\
\hline
1/4     &9.0102e+00    &1.3856e+01     &1.2526e+02                &6.1160e+01    &1.1121e+02   \\    \hline
1/8     &3.5126e+00    &3.2300e+00     &3.1002e+01                &1.3162e+01    &2.6948e+01 \\     \hline
1/16    &9.7594e-01    &8.0214e-01     &7.9261e+00                &3.4414e+00    &6.8899e+00 \\      \hline
1/32    &2.5123e-01    &2.0189e-01     &2.0219e+00                &9.3080e-01    &1.7653e+00  \\       \hline
1/64    &6.3761e-02    &5.0946e-02     &5.1331e-01                &2.4852e-01    &4.5010e-01      \\           \hline
1/128   &1.6049e-02    &1.2826e-02     &1.2945e-01                &6.4400e-02    &1.1380e-01  \\           \hline
1/256   &4.0207e-03    &3.2163e-03     &3.2468e-02                &1.6320e-02    &2.8570e-02\\           \hline
\hline
Rate      &2.00        &2.00         &2.00                         & 1.98     & 1.99 \\
\hline
\end{tabular}
\end{center}
}}
}
%\label{tab1-06}
%\end{small}
\end{table}

{\bf Test Case 9 (Nonhomogeneous BVP):} In this numerical test, the domain $\Omega=(0,1)^{2}$ is divided into two subdomains by the vertical line $x=\frac{1}{2}$. The diffusive coefficient tensor is the identity matrix $a=I$ for $x<0.5$ and $a=[10,3;3,1]$ for $x\geq 0.5$. The exact solution for this test problem is given by $u=1-2y^{2}+4xy+6x+2y$ for $x<0.5$ and $u=-2y^{2}+1.6xy-0.6x+3.2y+4.3$ for $x\geq 0.5$. This test problem has been considered in \cite{WG-CFO-Wang}. The numerical results are illustrated in Table \ref{Example6:rho1:RECWP}.

\begin{table}[htbp]\centering\scriptsize
{\color{black}{\caption{Test Case 9: Convergence of the lowest order WG-FEM on the $(0,1)^{2}$, discontinuous diffusive tensor, uniform rectangular partitions, stabilization parameter $\rho=1$, and $L^2$ projection of the boundary data.} \label{Example6:rho1:RECWP}
{
\setlength{\extrarowheight}{1.5pt}
\begin{center}
\begin{tabular}{|l|l|l|l|l|l|}
\hline $h$ &$\|u-\S(u_{b})\|_{\infty,\star}$& $\|u-\S(u_{b})\|_0$ & $\|\nabla_d({e_{b}})\|_0$  &$\|\nabla_d u_b-\nabla u\|_{0,\star}$  &    $\|\nabla(Q_{0}u-\S(u_{b}))\|_0$  \\
\hline
1/4     &8.3486e-02    &5.0480e-02     &2.7273e-01                &2.7273e-01    &2.7273e-01   \\    \hline
1/8     &2.1292e-02    &1.3308e-02     &8.4055e-02                &8.4055e-02    &8.4055e-02 \\     \hline
1/16    &5.3305e-03    &3.3745e-03     &2.4203e-02                &2.4203e-02    &2.4203e-02 \\      \hline
1/32    &1.3319e-03    &8.4796e-04     &6.7643e-03                &6.7643e-03    &6.7643e-03  \\       \hline
1/64    &3.3280e-04    &2.1252e-04     &1.8638e-03                &1.8638e-03    &1.8638e-03      \\           \hline
1/128   &8.3166e-05    &5.3180e-05     &5.0695e-04                &5.0695e-04    &5.0695e-04  \\           \hline
1/256   &2.0787e-05    &1.3299e-05     &1.3635e-04                &1.3635e-04    &1.3635e-04\\           \hline
\hline
Rate      &2.00        &2.00         &1.89                        & 1.89    & 1.89 \\
\hline
\end{tabular}
\end{center}
}}
}

%\label{tab1-06}
%\end{small}
\end{table}

\subsection{Numerical experiments with variable coefficients}\label{subSection:NE3} The last part of the numerical experiments shall consider model problems with variable diffusive coefficients.

{\bf Test Case 10 (Nonhomogeneous BVP):} The model problem in this test has domain $\Omega=(0,1)^{2}$ with exact solution $u=\sin(x)\sin(y)$. The diffusive coefficients are given by $a_{11}=1+\exp(y)$, $a_{12}=a_{21}=0.5$, and $a_{22}=1+\exp(x)$. Table \ref{Example6:rho1:RectWP} illustrates the corresponding numerical results arising from the WG finite element method. It can be seen that the numerical results outperform the superconvergence theory established in the last section.

\begin{table}[h]\centering\scriptsize
{\color{black}{\caption{Test Case 10: Convergence of the lowest order WG-FEM on  $\Omega=(0,1)^{2}$ with exact solution $u=\sin(x)\sin(y)$, uniform rectangular partitions, stabilization parameter $\rho=1$, and $L^2$ projection of the boundary data. The coefficient matrix is $a_{11}=1+\exp(y)$, $a_{12}=a_{21}=0.5$, and $a_{22}=1+\exp(x)$.}\label{Example6:rho1:RectWP}
{
\setlength{\extrarowheight}{1.5pt}
\begin{center}
\begin{tabular}{|l|l|l|l|l|l|}
\hline $h$ &$\|u-\S(u_{b})\|_{\infty,\star}$& $\|u-\S(u_{b})\|_0$ & $\|\nabla_d({e_{b}})\|_0$  &$\|\nabla_d u_b-\nabla u\|_{0,\star}$  &    $\|\nabla(Q_{0}u-\S(u_{b}))\|_0$  \\
\hline
1/4     &5.3731e-03    &3.0286e-03     &9.1755e-03                &9.5462e-03    &9.2227e-03   \\    \hline
1/8     &2.1932e-03    &8.4258e-04     &4.2623e-03                &4.3225e-03    &4.2726e-03 \\     \hline
1/16    &7.3908e-04    &2.2727e-04     &1.4827e-03                &1.4946e-03    &1.4850e-03 \\      \hline
1/32    &2.1165e-04    &5.8834e-05     &4.5247e-04                &4.5499e-04    &4.5297e-04  \\       \hline
1/64    &5.5959e-05    &1.4902e-05     &1.3016e-04                &1.3071e-04    &1.3027e-04      \\           \hline
1/128   &1.4294e-05    &3.7418e-06     &3.6290e-05                &3.6414e-05    &3.6315e-05  \\           \hline
1/256   &3.6091e-06    &9.3676e-07     &9.9215e-06                &9.9500e-06    &9.9273e-06\\           \hline
\hline
Rate      &1.99        &2.00         &1.87                       & 1.87     & 1.87 \\
\hline
\end{tabular}
\end{center}
}}
}

%\label{tab1-06}
%\end{small}
\end{table}

{\bf Test Case 11 (Reaction-diffusion equation):}\ This test case is concerned with the following reaction-diffusion equation: {\em Find an unknown function $u$ satisfying
\begin{eqnarray}%\label{model}
-\nabla \cdot (a\nabla u)+c u &=& f, \quad \mbox{in}~~ \O=(0,1)^{2}, \label{re-D-a1}\\
u &=& g,\quad \mbox{on}~~ \pa\O , \label{re-D-aa1}
\end{eqnarray}
where $a_{11}=1+\exp(2x)+y^{3}$, $a_{12}=a_{21}=\exp(x+y)$, $a_{22}=1+\exp(2y)+x^{3}$, and $c=2+x+y$}. Table \ref{Example7:rho1:Rect-di-WP} contains some numerical results for the problem (\ref{re-D-a1})-(\ref{re-D-aa1}) with exact solution $u=2\sin(2\pi x)\sin(3\pi y)$. Once again, the numerical results show an outstanding computational performance of the WG finite element method.

\begin{table}[h]\centering\scriptsize
{\color{black}{\caption{Test Case 11: Convergence of the lowest order WG-FEM on the $(0,1)^{2}$ with exact solution $u=2\sin(2\pi x)\sin(3\pi y)$, uniform rectangular partitions, stabilization parameter $\rho=1$, and $L^2$ projection of the boundary data. The coefficient matrix has entries  $a_{11}=1+\exp(2x)+y^{3}$, $a_{12}=a_{21}=\exp(x+y)$, and $a_{22}=1+\exp(2y)+x^{3}$.}\label{Example7:rho1:Rect-di-WP}
{
\setlength{\extrarowheight}{1.5pt}
\begin{center}
\begin{tabular}{|l|l|l|l|l|l|}
\hline $h$ &$\|u-\S(u_{b})\|_{\infty,\star}$& $\|u-\S(u_{b})\|_0$ & $\|\nabla_d({e_{b}})\|_0$  &$\|\nabla_d u_b-\nabla u\|_{0,\star}$  &    $\|\nabla(Q_{0}u-\S(u_{b}))\|_0$  \\
\hline
1/4     &7.9992e-01    &6.4835e-01     &7.2555e+00                &6.5731e+00    &7.0836e+00   \\    \hline
1/8     &5.2218e-01    &2.3675e-01     &3.5307e+00                &3.3640e+00    &3.4912e+00 \\     \hline
1/16    &2.2828e-01    &8.4924e-02     &1.4659e+00                &1.4322e+00    &1.4581e+00 \\      \hline
1/32    &7.2101e-02    &2.6669e-02     &4.9367e-01                &4.8678e-01    &4.9209e-01  \\       \hline
1/64    &1.9519e-02    &7.3441e-03     &1.4076e-01                &1.3921e-01    &1.4040e-01      \\           \hline
1/128   &4.9895e-03    &1.8941e-03     &3.7186e-02                &3.6817e-02    &3.7101e-02  \\           \hline
1/256   &1.2542e-03    &4.7772e-04     &9.5720e-03                &9.4822e-03    &9.5514e-03\\           \hline
\hline
Rate      &1.99        &1.99         &1.96                       & 1.96     & 1.96 \\
\hline
\end{tabular}
\end{center}
}}
}

%\label{tab1-06}
%\end{small}
\end{table}

\bigskip
\bigskip

\vfill\eject

\end{document}